\documentclass[11pt,letterpaper]{amsart}
\usepackage{amscd,amsmath,amssymb,euscript}
\usepackage[frame,cmtip,curve,arrow,matrix,line,graph]{xy}
\usepackage{tikz-cd}
\usepackage{bigints}
\usetikzlibrary{matrix}
\usepackage{hyperref}
\usepackage{stmaryrd}

\title[NCRs and IH for quotient singularities]{Noncommutative resolutions and intersection cohomology for quotient singularities}

\author{Tudor P\u adurariu}
\address{Department of Mathematics, Columbia University, 
2990 Broadway, New York, NY 10027}
\email{tgp2109@columbia.edu}

\newtheorem{thm}{Theorem}[section]
\newtheorem{cor}[thm]{Corollary}
\newtheorem{prop}[thm]{Proposition}

\theoremstyle{definition}

\newtheorem{thm*}[thm]{Theorem$^*$}

\newcommand{\comment}[1]{}

\renewcommand{\leq}{\leqslant}
\renewcommand{\geq}{\geqslant}

\newcommand{\D}{\mathbb{D}}
\newcommand{\F}{\mathcal{F}}

\newcommand{\OO}{\mathcal{O}}

\newcommand{\Z}{\mathcal{Z}}

\newcommand{\Sa}{\mathcal{S}}
\newcommand{\Y}{\mathcal{Y}}
\newcommand{\HH}{\mathrm{H}}
\newcommand{\IH}{\mathrm{IH}}
\newcommand{\KK}{\mathrm{K}}
\newcommand{\IK}{\mathrm{IK}}
\newcommand{\IC}{\mathrm{IC}}
\newcommand{\HP}{\mathrm{HP}}

\newcommand{\X}{\mathcal{X}}

\newcommand{\C}{\mathbb{C}}

\begin{document}
\maketitle

\begin{abstract}
    For a large class of good moduli spaces $X$ of symmetric stacks $\mathcal{X}$,
    we define noncommutative motives $\mathbb{D}^{\text{nc}}(X)$ which can be regarded as categorifications of the intersection cohomology of $X$. These motives are summands of noncommutative resolutions of singularities $\mathbb{D}(X)\subset D^b(\X)$ of $X$.  The category $\mathbb{D}(X)$ is a global analogue of
    the noncommutative resolutions of singularities of $V\sslash G$ for $V$ a symmetric representation of a reductive group $G$ constructed by \v{S}penko--Van den Bergh.
    
\end{abstract}


\section{Introduction}

\subsection{Intersection cohomology}  Intersection cohomology $\IH^{\cdot}(X, \mathbb{Q})$ and the BBDG Decomposition Theorem \cite{BBD} are important tools in the study of the topology of algebraic varieties. They have many applications in representation theory, see \cite{L}, or more recently in construction of representation of $\mathcal{W}$-algebras due to Braverman--Finkelberg--Nakajima \cite{BFN} and in the study of (Kontsevich--Soibelman) Cohomological Hall algebras \cite{COHA}, \cite{DM}. 
It is an important problem to define a K-theoretic version of intersection cohomology and it is expected that such a theory will have applications in representation theory. 
More generally, one can try to construct categorifications of intersection cohomology.

\subsection{Noncommutative resolutions}\label{nonres}
A natural place to look for such categorifications is inside noncommutative resolutions (NCRs) of $X$. There are more NCRs than usual resolution of singularities. 
A conjecture of Bondal--Orlov \cite[Section 5]{BO} says that there exists a \textit{minimal} NCR, i.e. a category admissible inside any other NCR of $X$. 

We look for natural candidates of minimal NCRs $\mathbb{D}$ of good moduli spaces $X$ of Artin stacks $\X$. Inspired by the Decomposition Theorem, we expect the periodic cyclic homology $\text{HP}_{\cdot}\left(\mathbb{D}\right)$ to have $\bigoplus_{i\in\mathbb{Z}}\IH^{\cdot+2i}(X, \mathbb{C})$ as a direct summand. We then try to construct noncommutative motives $\mathbb{D}^{\text{nc}}$ inside $\mathbb{D}$ whose periodic cyclic homology is $\bigoplus_{i\in\mathbb{Z}}\IH^{\cdot+2i}(X, \mathbb{C})$.


\subsection{NCRs for good moduli spaces}\label{assum}
We make three possible assumptions on a stack $\X$:

\textbf{A.} $\X$ is an algebraic stack locally of finite type over $\C$ with a good moduli space $\pi:\X\to X$ (as defined by Alper \cite{A}) such that $\pi$ has affine diagonal.

\textbf{B.} In addition to the above, assume that $\dim\,X=\dim\,\X$.

\textbf{C.} In addition to the above, assume that $\pi$ is generically an isomorphism.
\smallskip

Let $\X$ a smooth stack which satisfies Assumption A and let $p\in X(\mathbb{C})$.  By work of Alper--Hall--Rydh \cite[Theorem 1.2]{AHR}, 
there exists a smooth affine scheme $A$ with an action of a reductive group $G$ such that the following diagram is cartesian

\begin{equation}\label{d}
\begin{tikzcd}
\X \arrow[d, "\pi"'] & A/G \arrow[l, swap, "e"] \arrow[d]\\
X & A\sslash G, \arrow[l, swap, "e"]
\end{tikzcd}
\end{equation}
and $e$ is an \'etale map which contains $p$ in its image. The NCRs we are using are a global version of the NCRs constructed by \v{S}penko--Van den Bergh \cite{SvdB} for quotients $V/G$ for $V$ a representation of $G$.

\begin{thm}\label{thm1}
Let $\X$ be a symmetric stack satisfying Assumption B with good moduli space $X$. 
Then there exists a subcategory $\D(\X)$ admissible in $D^b(\X)$ which is an NCR of $X$. Its complement is generated by complexes supported on attracting stacks/ $\Theta$-stacks of $\X$.
\end{thm}

Assume that $\X=V/G$ where $V$ is a symmetric $G$-representation, i.e. the weights $\beta$ and $-\beta$ of $G$ have the same multiplicity in $V$. 
Let $M$ be the weight lattice of $G$, let $M_{\mathbb{R}}=M\otimes\mathbb{R}$, and let $W$ be the Weyl group of $G$. 
For $\lambda: \mathbb{C}^*\to G$ a cocharacter, define
\[n_{\lambda}:=\langle \lambda, \det\mathbb{L}_{\X}^{\lambda>0}\rangle,\]
where $\mathbb{L}_{\X}$ is the cotangent complex of $\X$.
Let $\delta\in M_{\mathbb{R}}^W$. The NCR constructed by \v{S}penko--Van den Bergh \cite{SvdB} is as follows: $\mathbb{D}_\delta(\X)$ is the full subcategory of $D^b(\X)$ generated by complexes $\mathcal{F}$ such that for any cocharacter $\lambda:\mathbb{C}^*\to G$:
\begin{equation}\label{bounds}
    -\frac{n_{\lambda}}{2}+\langle \lambda, \delta\rangle\leq \langle \lambda, \mathcal{F}|_0\rangle \leq \frac{n_{\lambda}}{2}+\langle \lambda, \delta\rangle.
    \end{equation}


Next, let $\mathcal{X}$ be as in Assumption B. Recall the description \eqref{d} for a point $p\in X$. Locally analytically near a point $x\in\X(\C)$, the stack $\X$ is isomorphic to an open analyic subset of $N/G$, where $G$ is the stabilizer of $x$ and $N$ is its normal bundle in $\X$. We assume that all such representations are symmetric, and call stacks $\X$ with this property \textit{symmetric}. Note that symmetric stacks are smooth.
Let $\ell\in \text{Pic}(\X)_{\mathbb{R}}$. For a morphism $\lambda:B\C^*\to \X$ with image in $\X(\C)$, define 
\[n_{\lambda}:=\text{wt}\,\lambda^*(\text{det}\,\mathbb{L}_{\mathcal{X}}^{\lambda>0}).\] 
The category $\mathbb{D}_\ell(\X)$ is the full subcategory of $D^b(\X)$ such that for any map $\lambda:B\C^*\to \X$ with image in $\X(\C)$, we have
\[-\frac{n_{\lambda}}{2}+\text{wt}\, \lambda^*\ell \leq  \text{wt}\,\lambda^*(\mathcal{F}) \leq \frac{n_{\lambda}}{2}+\text{wt}\, \lambda^*\ell.\]

To show $\mathbb{D}_\ell(\X)$ is a smooth and proper category over $X$, we show in Subsection \ref{ncr2} that $\D_\ell(\X)$ is admissible inside the Kirwan resolution of $\X$, a DM stack constructed by Edidin--Rydh \cite{ER} which recovers the Kirwan resolution in the local case \cite{Ki}. The analogous result for quotients $Y/G$ was proved by \v{S}penko--Van den Bergh \cite{SvdB3}. Our approach is different from theirs and uses the variation of GIT techniques from \cite{BFK}, \cite{HL}, \cite{HL2}. In order for $\pi^*: D^b(X)\to D^b(\X)$ to have image in $\mathbb{D}_\ell(\X)$, we choose $\ell$ to be trivial and drop it from the notation of $\mathbb{D}(\X)$. 

\subsection{Intersection cohomology of symmetric good moduli spaces}

Let $\X$ be a symmetric stack satisfying Assumption B from Subsection \ref{assum}. Define \begin{equation}\label{cohh}
    \textbf{B}:=\text{image}\,\left(\bigoplus_{\mathcal{S}} \HH^\cdot(\mathcal{Z})\xrightarrow{p_*q^*} \HH^\cdot(\X)\right),
    \end{equation}
over all attracting stacks $\mathcal{S}$ (different from $\X)$ with fixed stack and associated maps $\mathcal{Z}\xleftarrow{q}\mathcal{S}\xrightarrow{p}\X$ and the singular cohomology is with $\mathbb{Q}$ coefficients. We denote by $\textbf{P}^{\leq 0}\HH^\cdot(\X)$ the zeroth piece of the perverse filtration on $\HH^\cdot(\X)$ induced by the map $\pi:\X\to X$. Observe that if $\X$ satisfies Assumption C, then $\textbf{P}^{\leq 0}\HH^\cdot(\X)\cong\IH^\cdot(X)$.
In Section \ref{comp}, we show that:
\begin{thm}\label{thm12}
For $\X$ a symmetric stack satisfying Assumption B, there is a  decomposition
\[\HH^\cdot(\X)=\textbf{P}^{\leq 0}\HH^\cdot(\X)\oplus\textbf{B}.\]
\end{thm}
The above result follows from a version of the BBDG Decomposition Theorem for stacks, see Proposition \ref{bbdg}. The exceptional loci are covered by the attracting stacks, and by the symmetric assumption on $\X$, the images of classes from attracting stacks are in positive perverse degree. 

\subsection{Noncommutative motives for symmetric good moduli spaces}\label{abcd} Let $\X$ be a symmetric stack satisfying Assumption B from Subsection \ref{assum}.
In Section \ref{intKtheory} we define a noncommutative motive with rational coefficients $\D^{\text{nc}}(\X)=(\D(\X), e)$ where $e$ is an idempotent in $\KK_0\left(\text{rep}\left(\D(\X), \D(\X)\right)\right)_\mathbb{Q}$, see Subsection \ref{ncmotives} for a brief discussion of noncommutative motives.
Define 
\[\textbf{B}\KK_\cdot(\X):=\text{image}\,\left(\bigoplus_{\mathcal{S}} \KK_\cdot(\mathcal{Z})\xrightarrow{p_*q^*} \KK_\cdot(\X)\right),\]
where the sum is an in \eqref{cohh} and $\KK_\cdot$ is rational K-theory.
In Section \ref{cat}, we show an analogue on Theorem \ref{thm12} in K-theory:
\begin{thm}\label{thm13}
For $\X$ a symmetric stack satisfying Assumption B, there is a decompositions
\[\KK_\cdot(\X)=\KK_\cdot(\mathbb{D}^{\text{nc}}(\X))\oplus\textbf{B}\KK_\cdot(\X).\]
\end{thm}
The category $\D(\X)$ contains, in general, complexes on attracting stacks, but it may be indecomposable as a triangulated category. For $\lambda :B\mathbb{C}^*\to \X$ with associated fixed stack $\Z$, let $\D(\Z)_b$ be the subcategory of $\D(\Z)$ of complexes on which $\lambda$ acts with weight $b=\frac{n_{\lambda}}{2}$.
The motive $\D^{\text{nc}}(\X)$ is a complement of the images \[m_\lambda:=p_*q^*: \D(\Z)_b\rightarrow \D(\X)\] from fixed substacks $\Z$. By a result of Thomason \cite[Corollary 2.17]{Th}, one can compute rational K-theory using an \'etale cover, so it suffices to check the statement in the local case $\X=A/G$, where $A$ is a smooth affine scheme.
The main tool in understanding these images for different fixed substacks is a product-coproduct type compability that we briefly explain. Let $\lambda$ and $\mu$ be two dominant cocharacters, let $\textbf{S}$ be a set of cocharacters $\nu$ refining $\mu$ and $\lambda$, see Subsection \ref{S}. Then the following diagram commutes:
\begin{equation*}
\begin{tikzcd}
\KK_\cdot\left(\D(\X^\lambda)_b\right)\arrow[d, "\bigoplus_{\textbf{S}} \Delta^\lambda_\nu"']\arrow[r, "m_\lambda"]& \KK_\cdot(\D(\X))\arrow[d, "\Delta_\mu"]\\
\bigoplus_{\textbf{S}} \KK_\cdot\left(\D(\X^\nu)_b\right) \arrow[r, "\bigoplus_{\textbf{S}} \widetilde{m^\mu_{\nu}}"] & \KK_\cdot\left(\D(\X^\mu)_b\right).
\end{tikzcd}
\end{equation*}
Here, $\Delta$ are restriction maps and $\widetilde{m}$ denotes a twist of the multiplication. We check that the diagram commutes by a direct computation using shuffle formulas for $m$ and $\Delta$. 

\subsection{Intersection K-theory}
Let $\X$ be a symmetric stack satisfying Assumption C and let $X$ be its good moduli space. Then $\textbf{P}^{\leq 0}\HH^\cdot(\X)\cong\IH^\cdot(X)$. Given Theorems \ref{thm12} and \ref{thm13}, we propose to call $\KK_\cdot(\D^{\text{nc}}(\X))=:\IK_\cdot(X)$ \textit{the intersection K-theory of} $X$. 
There is a Chern character map
\[\text{ch}: \IK_\cdot(X)\to \IH^\cdot(X).\]
Further, from the construction of $\IK_\cdot(X)$, we obtain a natural surjection
\[\KK_\cdot(\X)\twoheadrightarrow \IK_\cdot(X).\]
The analogous statement in cohomology was proved by Kirwan \cite{Ki2}. 
Further, in Subsection \ref{Kirwa}, we show that if the Kirwan resolution is a scheme, then 
\[\HP_i\left(\mathbb{D}^{\text{nc}}(\X)\right)\cong \bigoplus_{j\in\mathbb{Z}} \IH^{i+2j}(X,\mathbb{C}).\]

We use the construction of K-theory of quiver varieties to prove a version of a Poincaré-Birkhoff-Witt Theorem for K-theoretic Hall algebras of quivers with potential \cite[Theorem 1.2]{P2}. For example, for a quiver $Q$, the K-theoretic Hall algebra $\text{KHA}(Q,0)$ for zero potential is generated by spaces which are twisted versions of $\IK_\cdot(X(d))$, where $X(d)$ is the coarse space of representation of $Q$ of dimension $d$.


\subsection{Previous work on intersection K-theory.}\label{previouswork}
There are other approaches of defining intersection K-theory in particular cases. Cautis \cite{C}, Cautis--Kamnitzer \cite{CK} have an approach for categorification of intersection sheaves for certain subvarieties of the affine Grassmannian.
For varieties with a cellular stratification, Eberhardt proposed a definition in \cite{E}. 

A related problem is defining intersection Chow groups. Corti--Hanamura \cite{CH1}, \cite{CH2} proposed two approaches towards intersection Chow groups for general varieties $X$, one which proves a version of the decomposition theorem under some conjectures.
In \cite{P2}, we propose a definition of intersection (graded) $\text{gr}^\cdot \KK$-theory of any singular variety $X$ which is a summand of $\text{gr}^\cdot \KK_\cdot(Y)$ for any resolution of singularities $Y\to X$; the associated graded is taken with respect to the codimension filtration, so its zeroth level is a version of intersection Chow groups.
De Cataldo--Migliorini \cite{dCM} proposed a definition of intersection Chow motive for singularities with a semismall resolution.



\subsection{Future directions}
The definitions of $\D^{\text{nc}}(\X)$ and $\IK_\cdot(X)$ are for symmetric stacks $\X$. It is worth trying to find analogues of these constructions beyond symmetric stacks. One idea is to look for symmetric substacks $\X'\subset \X$ whose complement is a union of $\Theta$-strata and thus have good moduli spaces $X'$. It is true that $\IH^\cdot(X)$ is a direct summand of $\IH^\cdot(X')$, but we do not know how to characterize the difference between them.



The discussion in Subsection \ref{nonres} serves as motivation for our work, but we do not make progress towards the Bondal--Orlov conjecture. 
In general, $\mathbb{D}(\X)$ and $\mathbb{D}^{\text{nc}}(\X)$ are different. However, in the cases in which they are equal, for example for quotients $V/T$ where $T$ is a torus and $V$ is a $T$-representation, 
it is natural to guess that $\mathbb{D}(V/T)$ is minimal in the sense of Bondal--Orlov. 
The category $\D(V/T)$ is indecomposable \cite[Appendix A]{SvdB2}, so it should be minimal if the Bondal--Orlov conjecture is true for this particular class of singularities. 

We plan to return to these questions in future work.

\subsection{Acknowledgements} I thank Davesh Maulik for numerous conversations related to the present paper. 
I thank the Institute of Advanced Studies for support during the preparation of the paper. This material is based upon work supported by the National Science Foundation under Grant No. DMS-1926686.

\section{Preliminaries}\label{prel}

\subsection{Conventions and notations}\label{conven}
All spaces considered are over $\mathbb{C}$. All schemes considered are of finite type over $\mathbb{C}$. All points considered are $\mathbb{C}$-points unless otherwise stated.

For $A$ a scheme with an action of a reductive group $G$, we denote the quotient stack by $A\slash G$ and its coarse space by $A\sslash G$.

For $X$ a scheme or stack, denote by  $\text{QCoh}\,(X)$ the category of (unbounded) complexes of quasicoherent sheaves on $X$, by 
$D^b\text{Coh}\,(X)$ its subcategory of compact objects, i.e. 
the derived category of bounded complexes of coherent sheaves, and by $\text{Perf}\,(X)$ the subcategory of $D^b\text{Coh}\,(X)$ of perfect complexes. 
The functors used in the paper are derived; we sometimes drop $R$ or $L$ from notation, for example we write $f_*$ instead of $Rf_*$. Denote by $D_{\text{shvs}}(X)$ the category of complexes of constructible sheaves on a space $X$. 

For $G$ a reductive group, fix a maximal torus and a Borel subgroup $T\subset B\subset G$. 
We denote by $M$ the character lattice, by $N$ the cocharacter lattice, by $W$ the Weyl group, by $M_\mathbb{R}:=M\otimes_\mathbb{Z}\mathbb{R}$, and by \[\langle\,,\,\rangle:N\times M\to \mathbb{Z}\] the natural pairing; it induces a pairing between $N$ and $M_\mathbb{R}$. We assume that the weights in the Lie algebra of $B$ are negative roots. In particular, $B$ induces a choice of a dominant chamber $M_\mathbb{R}^+\subset M_\mathbb{R}$. Denote by $\rho$ half the sum of positive roots of $G$. For $\chi$ a dominant weight of $G$, denote by $\Gamma(\chi)$ the irreducible representation of $G$ of highest weight $\chi$.
By abuse of notation, for $V$ a representation of $G$, write
\[\langle \lambda, V\rangle:=\langle \lambda, \det V\rangle=
\Big\langle \lambda, \sum_{\beta\text{ wt of }V}\beta\Big\rangle.\]
We denote by $w\cdot \chi$ the usual $W$-action and by $w*\chi=w(\chi+\rho)-\rho$ the shifted $W$-action. We denote by $\mathfrak{g}$ the Lie algebra of $G$.

For a stack $\X$, we denote by $\HH^\cdot(\X)$ the singular cohomology with $\mathbb{Q}$ coefficients and by $\KK_\cdot(\X)$ the rational K-theory of $\X$.

\subsection{Semi-orthogonal decompositions and noncommutative resolutions}
In this Subsection, we recall some basic notions related to derived categories. References for this material are \cite{K1}, \cite{K3}. 

\subsubsection{}\label{adjoint}
Let $\mathcal{T}$ be a triangulated category, and let $\mathcal{A}_1,\cdots, \mathcal{A}_n\subset \mathcal{T}$ be triangulated subcategories. We say that there is a \textit{semi-orthogonal decomposition}
    \[\mathcal{T}=\langle \mathcal{A}_1,\cdots, \mathcal{A}_n \rangle\] if the following two conditions are satisfied:
    
    (i) $\text{Hom}(A_i, A_j)=0$ for all $A_i\in\mathcal{A}_i, A_j\in\mathcal{A}_j$ and $1\leq j<i\leq n$.
    
    (ii) the smallest triangulated subcategory of $\mathcal{T}$ containing all $\mathcal{A}_i$s is $\mathcal{T}$.

\subsubsection{}\label{rightfunctor}
Let $\mathcal{T}$ be a triangulated category, and let $\mathcal{A}$ be a subcategory. $\mathcal{A}$ is called \textit{right admissible} in $\mathcal{T}$ if there exists a semi-orthogonal decomposition $\mathcal{T}=\langle-, \mathcal{A}\rangle$. Equivalently, the inclusion functor $\mathcal{A}\hookrightarrow \mathcal{T}$ has a right adjoint.

\subsubsection{}
In this paper, we say that a triangulated category $\mathcal{D}$ is \textit{smooth} if it is admissible inside $D^b(Y)$ for a smooth DM stack $Y$. We say that $\mathcal{D}$ is \textit{proper} over a given variety $S$ if it is admissible in $D^b(T)$ for $T$ a proper DM stack over $S$. 

\subsubsection{}

Let $X$ be a variety. We say that a smooth and proper over $X$ triangulated category $\mathcal{D}$ is a \textit{noncommutative resolution of singularities (NCR)} if there exists an adjoint pair of functors
    \begin{align*}
    F&: \mathcal{D}\to D^b(X),\\ 
    G&: \text{Perf}\,(X)\to \mathcal{D}
    \end{align*} such that $FG=\text{id}_{\text{Perf}(X)}$.
The definition is slightly more general than the definition in \cite[Definition 3.1 and the paragraph after it]{K3}.

\textbf{Example.} Let $X$ be a variety with rational singularities and let $f:Y\to X$ be a resolution of singularities. Then $D^b(Y)$ is an NCR of $X$ where the corresponding functors are $f_*:D^b(Y)\to D^b(X)$ and $f^*:\text{Perf}\,(X)\to D^b(Y)$.

\subsection{Good moduli spaces}\label{goodmoduli}
\subsubsection{} Let $\X$ be a stack. An algebraic space $X$ with a morphism $\pi:\X\to X$ is called a \textit{good moduli space} if

(i) $\pi_*: \text{QCoh}\,(\X)\to\text{QCoh}\,(X)$ is exact,

(ii) the induced morphism $\OO_X\to \pi_*\OO_\X$ is an isomorphism.
\\

\textbf{Examples.} \begin{enumerate}\item 
    For $A$ an affine variety and $G$ a reductive group acting linearly on $A$, \[\pi:A/G\to A\sslash G=\text{Spec}\left(\OO_A^G\right)\] is a good moduli space.
    
\item 
Let $Y$ be a smooth projective variety, let $\beta\in H^\cdot (Y)$, and let $\mathcal{L}$ be an ample line bundle on $Y$. The moduli stack $\mathfrak{M}^{\text{ss}}_{\beta}$ of $\mathcal{L}$-Gieseker semistable sheaves with Chern character $\beta$ has a good moduli space $M^{\text{ss}}_{\beta}$.

\end{enumerate} 

For properties and further examples of good moduli space, see Alper \cite{A}. We assume that $X$ is a scheme in this paper. The following is \cite[Theorem $4.12$, Theorem $1.2$]{AHR}:

\begin{thm}\label{ahr}
Let $\X$ be a stack satisfying Assumption A. Let $x\in \X(\mathbb{C})$ with stabilizer group $G_x$ and normal bundle $N_x$. Then there exists an affine scheme $A$ with a linearizable action of $G_x$, a point $a\in A(\C)$ with stabilizer group $G_x$, and \'etale maps $e$ and $f$ such that the following squares are cartesian:
\begin{equation}\label{d1}
\begin{tikzcd}
(N_x/G_x, 0) \arrow[d] & (A/G_x, a) \arrow[r, "e"] \arrow[l, swap, "f"] \arrow[d]& (\X,x) \arrow[d]  \\
N_x\sslash G_x &A\sslash G_x \arrow[l, swap, "f"]
\arrow[r, "e"]& X.
\end{tikzcd}
\end{equation}
\end{thm}

We will be using the following corollary:

\begin{cor}\label{ahr2}
Let $\pi: \X\to X$ be a stack satisfying Assumption A. Let $p\in X(\mathbb{C})$. There exists a quotient stack $p: \Y:=V/G\to Y:=V\sslash G$ with $G$ a reductive group, $V$ a $G$-representation, and analytic open sets $p\in U\subset X$ and $0\in \mathcal{U}\subset Y$ such that the diagram 
\begin{equation}\label{d2}
\begin{tikzcd}
\pi^{-1}(U) \arrow[d, "\pi"'] & p^{-1}(\mathcal{U}) \arrow[l, swap, "\sim"] \arrow[d, "p"]\\
U & \mathcal{U}. \arrow[l, swap, "\sim"]
\end{tikzcd}
\end{equation}
commutes. In particular, the diagram \begin{equation}\label{d3}
\begin{tikzcd}
\widehat{\X_p} \arrow[d, "\pi"'] & \widehat{\Y_0} \arrow[l, swap, "\sim"] \arrow[d, "p"]\\
\widehat{X_p} & \widehat{Y_0}, \arrow[l, swap, "\sim"]
\end{tikzcd}
\end{equation}
where $\widehat{\X_p}$ is the formal completion of $\X$ along $\pi^{-1}(p)$, $\widehat{X_p}$ is formal completion of $X$ at $p$ etc.
\end{cor}


\subsection{Theta-stratifications}

References for this Section are \cite[Section 1]{HL}, \cite[Section 2.1]{HL2}, \cite{AHLH}.

\subsubsection{}\label{theta} Let $\X$ be an algebraic stack of finite type over $\C$. Let $\Theta=\mathbb{A}^1/\mathbb{G}_m$. The stacks $\underline{\text{Map}}(B\mathbb{G}_m,\X)$ and 
$\underline{\text{Map}}(\Theta,\X)$ are algebraic stacks with natural (evaluation) maps to $\X$. Their connected components are called \textit{fixed stacks} and $\Theta$\textit{-stacks} or \textit{attracting stacks}, respectively. 
There is also a natural map \[\underline{\text{Map}}(\Theta,\X)\to \underline{\text{Map}}(B\mathbb{G}_m,\X).\] 
A $\Theta$-stack $\mathcal{S}$ has an associated fixed stack $\mathcal{Z}$, and fits in a diagram
\begin{equation*}
    \begin{tikzcd}
    \mathcal{S}\arrow[d,"q"] \arrow[r,"p"]& \mathcal{X}\\
    \mathcal{Z}& 
    \end{tikzcd}
\end{equation*}
where $p$ is proper and $q$ is an affine bundle map. If $p$ is an immersion, we say that $\mathcal{S}$ is a $\Theta$\textit{-stratum}.

When $\X=V/G$ for $V$ an representation a reductive group $G$, the fixed and $\Theta$-stacks are of the form
\begin{equation}\label{d4}
    \begin{tikzcd}
    V^{\lambda\geq 0}\slash G^{\lambda\geq 0}\arrow[d,"q"] \arrow[r,"p"]& V\slash G\\
    V^\lambda\slash G^\lambda& 
    \end{tikzcd}
\end{equation}
where $\lambda:\C^*\to G$ is a cocharacter, $G^\lambda$ is the Levi group associated to $\lambda$, $G^{\lambda\geq 0}$ is the parabolic group associated to $\lambda$, $V^\lambda\subset V$ is the $\lambda$-fixed locus and $V^{\lambda\geq 0}\subset V$ is the $\lambda$-attracting locus. Such a $\Theta$-stack is a $\Theta$-stratum if the map $p$ is a closed immersion, so if it is a Kempf--Ness locus in the terminology of \cite[Section 2.1]{HL2}.





\subsection{Noncommutative motives}\label{ncmotives}

We briefly explain the definition of noncommutative motives. A general reference is \cite{T}. 
Denote by $\textbf{dgcat}$ the category of (small) dg categories over $\mathbb{C}$. It has a Quillen model structure whose weak equivalences are derived Morita equivalences. Denote by $\textbf{Hmo}$ the corresponding homotopy category. 
The universal category through which all additive invariants factor (examples include cyclic homology, K-theory, and related constructions) is a smaller (additive) category $\textbf{Hmo}_0$ with objects dg categories and morphisms \[\text{Hom}_{\textbf{Hmo}_0}(\mathcal{A}, \mathcal{B})=\KK_0\left(\text{rep}(\mathcal{A}, \mathcal{B})\right)\]
where $\mathcal{A}$ and $\mathcal{B}$ are dg categories and $\text{rep}(\mathcal{A}, \mathcal{B})\subset D^b(\mathcal{A}^{\text{op}}\otimes\mathcal{B})$ is the full subcategory of bimodules $X$ such that $X(a,-)\in \text{Perf}\,(\mathcal{B})$ for any object $a\in\mathcal{A}$. Consider the functor $\mathcal{U}: \textbf{dgcat}\to \textbf{Hmo}_0$.


We consider the category $\textbf{Hmo}_{0;\mathbb{Q}}^{\natural}$, the idempotent completion of the $\mathbb{Q}$-linearization of 
$\textbf{Hmo}_0$. We call its elements (by a slight abuse) noncommutative motives; the original definition considers the subcategory of $\textbf{Hmo}_{0;\mathbb{Q}}^{\natural}$ generated by proper and smooth dg categories, but in our case we need to allow proper and smooth categories over (not necessarily proper) $X$.

\subsection{A preliminary result}

The following type of result in used by \v{S}penko--Van den Bergh in their construction of NCRs. 

Let $A$ be a smooth affine variety with an action of a reductive group $G$ and let $\X=A/G$. For a locally free sheaf $\mathcal{F}$ on $\X$, its stalk at the origin is a representation $\Gamma$ of $G$. We call $\Gamma$ \textit{the associated representation of} $\mathcal{F}$.

We state the following result for future reference, the proof is same as \cite[Section 3.2]{HLS} and it uses an explicit K\"{o}szul resolution for pushforward along the map $A^{\lambda\geq 0}/G^{\lambda\geq 0}\hookrightarrow A/G^{\lambda\geq 0}$ and the Borel-Bott-Weil Theorem for the map $A/G^{\lambda\geq 0}\to A/G$.

\begin{prop}\label{bbw}
Let $\lambda$ be a cocharacter of $G$ and consider the diagram of attracting loci
\[A^\lambda/G^\lambda\xleftarrow{q}
A^{\lambda\geq 0}/G^{\lambda\geq 0}\xrightarrow{p} A/G.\]
Let $\F$ be a locally free sheaf on $A^\lambda/G^\lambda$ with associated representation $\Gamma(\chi)$ where $\chi$ is a dominant weight of $G$. 
Then there is a complex
\[\left(\bigoplus_{I}\mathcal{F}_I[|I|-\ell(w)], d\right)\to p_*q^*\mathcal{F},\]
where the terms of the complex correspond to subsets $I\subset \{\beta|\,\langle \lambda, \beta\rangle<0\}$, $\mathcal{F}_I$ is a locally free sheaf with associated representation \[\Gamma\left((\chi-\sigma_I)^+\right),\] where
$\sigma_I=\sum_{\beta\in I} \beta$,
$(\chi-\sigma_I)^+$ is the dominant Weyl-shifted conjugate of $\chi-\sigma_I$ if it exists, and zero otherwise,
and
$w$ is the element of the Weyl group such that $w*(\chi-\sigma_I)$ is dominant or zero of length $\ell(w)$.
\end{prop}

\section{Noncommutative resolutions of quotient singularities}\label{ncr}

In this Section, we prove Theorem \ref{thm1}. The definition of the categories $\mathbb{D}_\ell(\X)$ are for symmetric stacks satisfying Assumption A. In order to obtain NCRs, we need to assume that $\X$ satisfies Assumption C.
Recall the construction of category $\mathbb{D}(\X)$ and the strategy of proof discussed in Subsection \ref{assum}. 
\medskip 

\subsection{Local case} 
Let $\X=V/G$ where $V$ is a symmetric $G$-representation. Denote by $X=V\sslash G$ and by $p:\X\to X$. We will use the notations from Subsection \ref{conven}.
For a cocharacter $\lambda:\mathbb{C}^*\to G$, recall the diagram of attracting loci \eqref{d3} and define 
\[n_{\lambda}:=\langle \lambda, V^{\lambda>0}\rangle-\langle \lambda, \mathfrak{g}^{\lambda>0}\rangle=\langle \lambda, \det\mathbb{L}_{p}\rangle=\langle \lambda, \det\mathbb{L}_{\X}^{\lambda>0}\rangle.\]
Let $\delta\in M_{\mathbb{R}}^W$, and
let $\mathbb{D}_\delta(\X)$ be the full subcategory of $D^b(\X)$ generated by complexes $\mathcal{F}$ such that for any cocharacter $\lambda:\mathbb{C}^*\to G$:
\[-\frac{n_{\lambda}}{2}+\langle \lambda, \delta\rangle\leq \langle \lambda, \mathcal{F}|_0\rangle \leq \frac{n_{\lambda}}{2}+\langle \lambda, \delta\rangle.\] 
Let $\mathbb{W}\subset M_\mathbb{R}$ be the polytope 
\begin{equation}\label{defW}
    \mathbb{W}:=\text{sum}\,[0,\beta]\subset M_{\mathbb{R}},
\end{equation}
where the Minkowski sum is taken over all weights $\beta$ of $V$. 
The category $\mathbb{D}_\delta(\X)$ can be also described as the full subcategory of $D^b(\X)$ generated by vector bundles $\OO_{\X}\otimes \Gamma(\chi)$ where $\chi$ is a dominant weight of $G$ such that
\[\chi+\rho+\delta\in \frac{1}{2}\mathbb{W},\]
where the sum is taken over all weights $\beta$ of $V$.

Define $\mathbb{A}_\delta$ as the subcategory of $D^b(\X)$ generated by complexes $p_{\lambda*}q_\lambda^*(\mathcal{E})$ where $\mathcal{E}$ is a complex in $D^b(\mathcal{X}^\lambda)$ with 
\[\langle \lambda, \mathcal{E}|_0\rangle<-\frac{n_\lambda}{2}+\langle\lambda, \delta\rangle.\]
The following was proved by \v{S}penko--Van den Bergh \cite[Proposition 8.4]{SvdB} (the semi-orthogonal decomposition in loc. cit. holds for quotient stacks satisfying Assumption A, the condition that $X$ has a $T$-fixed stable point in loc. cit. is necessary to identify the summands with NCRs): 

\begin{thm}\label{local}
There exists a semi-orthogonal decomposition 
\[D^b(\X)=\langle \mathbb{A}_\delta, \mathbb{D}_\delta\rangle.\]
The semi-orthogonal decomposition holds relative to $X$ in the following sense: if $\mathcal{A}\in \mathbb{A}_\delta$ and $\mathcal{D}\in \mathbb{D}_\delta$, then \[Rp_*\left(R\mathcal{H}om_{\X}(\mathcal{D}, \mathcal{A})\right)=0.\]
\end{thm}

\subsection{Global case} Assume that $\pi: \X\to X$ is a symmetric stack satisfying Assumption B. Let $\ell\in \text{Pic}(\X)_{\mathbb{R}}$. 
Recall the definition of $\D_\ell(\X)$ from Subsection \ref{assum}. 
As in Step $1$, define $\mathbb{A}_\ell$ as the subcategory of $D^b(\X)$ generated by complexes of sheaves $p_*q^*\left(\mathcal{E}\right)$, where $\mathcal{S}$ is a $\Theta$-stack, $\mathcal{Z}$ is its associated fixed stack with maps
$\mathcal{Z}\xleftarrow{q}\mathcal{S}\xrightarrow{p}\X,$
and $\mathcal{E}$ satisfies \[\text{wt}\,\lambda^*(\mathcal{E})<-\frac{n_{\lambda}}{2}+\langle\lambda, \delta\rangle.\]

\begin{thm}\label{symsod}
There is a semi-orthogonal decomposition
\[D^b(\X)=\langle \mathbb{A}_\ell, \mathbb{D}_\ell\rangle.\]
The semi-orthogonal decomposition holds relative to $X$ in the following sense: if $\mathcal{A}\in \mathbb{A}_\ell$ and $\mathcal{D}\in \mathbb{D}_\ell$, then \begin{equation}\label{ad}
    R\pi_*\left(R\mathcal{H}om_{\X}(\mathcal{D}, \mathcal{A})\right)=0.\end{equation}
\end{thm}

For $p\in X$, let $D^b_o\left(\widehat{\X_p}\right)$ be the split category generated by the restrictions of complexes in $D^b(\X)$. For $U\subset X$ as in Corollary \ref{ahr2}, let $\X_U:=\pi^{-1}(U)$ and define $D^b_o\left(\X_U\right)$ as the split category of coherent analytic sheaves generate by restrictions of complexes in $D^b(\X)$. For $\mathcal{Y}=V/G$, the category $D^b\left(\mathcal{Y}\right)$ is generated by $\mathcal{O}_\Y\otimes \Gamma(\chi)$ for $\chi$ a dominant weight of $G$, and thus
$D^b\left(\widehat{\mathcal{Y}_0}\right)$ is generated 
by $\mathcal{O}_{\widehat{\Y}}\otimes \Gamma(\chi)$. By Theorem \ref{ahr}, the categories  $D^b_o\left(\X_U\right)$ and $D^b_o\left(\widehat{\X_p}\right)$ are also generated by these vector bundles. 

Further, let $\widehat{\mathbb{D}_{\ell, p}}\subset D^b_o\left(\widehat{\X_p}\right)$ be the category generated by restrictions of sheaves in $\mathbb{D}_\ell$; we define $\widehat{\mathbb{A}_{\ell, p}}$, $\widehat{\mathbb{D}_{\ell, U}}$ etc. similarly. For the stack $\mathcal{Y}=V/G$ from Corollary \ref{ahr2}, let $\delta\in M(G)_{\mathbb{R}}$ be the restriction of $\ell$ and
we denote by $\mathbb{A}'_\delta$ and $\mathbb{D}'_\delta$ the categories from Theorem \ref{local}.
By Corollary \ref{ahr2} and using the notation from there, we have that \begin{align*}
    \mathbb{D}_{\ell, U}&\cong \mathbb{D}'_{\delta, \mathcal{U}},\\
    \mathbb{A}_{\ell, U}&\cong \mathbb{A}'_{\delta, \mathcal{U}}
\end{align*}
and the analogues equivalences for formal completions. 

\begin{proof}[Proof of Theorem \ref{symsod}]
We continue with the notations from the above.
Let $\mathcal{A}\in \mathbb{A}_\ell$ and $\mathcal{D}\in \mathbb{D}_\ell$. To show \eqref{ad}, it suffices to prove the statement after restriction to $\widehat{\X_p}$ for all $p\in X$. Then $\mathcal{A}\in \widehat{\mathbb{A}'_{\delta, 0}}$ and $\mathcal{D}\in \widehat{\mathbb{D}'_{\delta, 0}}$ and thus the claim follows from Theorem \ref{local}. To show that $\mathbb{A}_\ell$ and $\mathbb{D}_\ell$ generate $D^b(\X)$, observe that they generate the local categories $D^b_o\left(\X_U\right)$. The claim follows as in \cite[Proposition 3.5.8]{SvdB}.
\end{proof}

\subsection{NCR}\label{ncr2} 
We show that $\mathbb{D}_\ell(\X)$ is a smooth and proper over $X$ category. More precisely,  $\mathbb{D}(\X)$ is an admissible subcategory of the Kirwan resolution of $\X$ as constructed by Edidin--Rydh \cite{ER}. In loc. cit., the authors do not use the language of $\Theta$-strata, but their construction is natural and applied to quotient stacks recovers Kirwan's resolution of singularities \cite[pages 475-476]{Ki}. 
For a stack $\X$ satisfying Assumption B,
the Edidin--Rydh construction provides a sequence of stacks \[\X=:\X_0\xleftarrow{\pi_0}\X_1\hookleftarrow\X^{\text{ss}}_1\xleftarrow{\pi_1}\cdots\xleftarrow{\pi_n}\X_{n+1}\hookleftarrow \X^{\text{ss}}_{n+1}=:Y,\]
with the following properties:

(i) the stacks $\X_i$ have good moduli spaces $X_i$,

(ii) $\X_i^{\text{ss}}\subset \X_i$ is an open subset, complement to $\Theta$-strata,

(iii) \'etale locally on $X_k$, either $\X_k\xleftarrow{\pi_k}\X_{k+1}\hookleftarrow\X^{\text{ss}}_{k+1}$ are isomorphisms, or
there are neighborhoods as in Theorem \ref{ahr}: 
\begin{equation*}
    \begin{tikzcd}
    X_k&\X_k\arrow[l]&\X_{k+1}\arrow[l, "\pi_k"']& \X^{\text{ss}}_{k+1}\arrow[l, hook']\\
    A\sslash G\arrow[u]\arrow[d]&
    A/G\arrow[l]\arrow[u]\arrow[d]& \left(\text{Bl}_{A^0}A\right)/G\arrow[u]\arrow[l, "\pi"']\arrow[d]& \left(\text{Bl}_{A^0}A\right)^{\text{ss}}/G \arrow[l, hook']\arrow[u]\arrow[d]\\
   N\sslash G& N/G\arrow[l, "\pi"']&\left(\text{Bl}_{N^0}N\right)/G\arrow[l, "\pi"']& \left(\text{Bl}_{N^0}N\right)^{\text{ss}}/G. \arrow[l, hook']
    \end{tikzcd}
\end{equation*}
Here, $A^0\subset A$ is the $G$-fixed locus, the stability condition is given by the tautological line bundle $\OO(1)$ on $\text{Bl}_{A^0}A$, and $N$ is the normal bundle of $a$ in $A$.

(iv) The stack $Y:=\X^{\text{ss}}_{n+1}$ is a smooth Deligne--Mumford stack with a proper map $Y\to X$ of relative dimension zero.
\medskip

The class $\ell\in \text{Pic}(\X)_\mathbb{R}$ induces classes which we also denote by $\ell\in \text{Pic}(\X_k)_\mathbb{R}$ for $1\leq k\leq n+1$.

\begin{prop}\label{indsym}
Let $0\leq k\leq n$.
Assume that $\X_k$ is a symmetric stack. Then $\X^{\text{ss}}_{k+1}$ is also a symmetric stack. 
\end{prop}

\begin{proof}
It suffices to check the statement in the local case
\[V/G\xleftarrow{\pi} \text{Bl}_0V/G\hookleftarrow (\text{Bl}_0V)^{\text{ss}}/G,\]
where $G$ is a reductive group, $V$ is a symmetric $G$-representation with $0$ the only $G$-fixed point in $V$, and the linearization is given by the tautological line bundle $\mathcal{O}(1)$. 

We claim that the unstable loci of $\text{Bl}_0V\cong \text{Tot}_{\mathbb{P}(V)}\left(\OO(-1)\right)$ are determined by pairs \[(\lambda, \mathbb{P}(V_a))\] where $\lambda:\C^*\to G$, $V_a\subset V$ is the subspace on which $\lambda$ acts with weight $a$, and $a<0$. 
The GIT algorithm, see \cite[Section 2.1]{HL2}, eliminates pairs $(\lambda, Z)$ for $\lambda:\C^*\to G$, $Z$ is a $\lambda$-fixed component on $\text{Tot}_{\mathbb{P}(V)}\left(\OO(-1)\right)$, and  \[\text{wt}_{\lambda}\,\mathcal{O}(1)|_Z>0.\]
The fixed loci $Z$ are $\mathbb{P}(V_a)$ for $a\neq 0$ and $\text{Tot}_{\mathbb{P}(V_0)}\left(\OO(-1)\right)$. 
Further, we compute
\[\text{wt}_{\lambda}\mathcal{O}(1)|_{\mathbb{P}(V_a)}=\sum_{n\in\mathbb{Z}} (n-a)\dim V_n>0.\]
The representation $V$ is symmetric, so $\text{dim}\,V_i=\text{dim}\,V_{-i}$ for $i\in\mathbb{Z}$.
We thus have that 
\[0<\sum_{n\in\mathbb{Z}} n\cdot\dim V_n-a\cdot\dim V=-a\cdot\dim V,\] so indeed $a<0$.

Finally, let $\lambda:\C^*\to G$ be a cocharacter which fixes a point $v$ in 
$\text{Tot}_{\mathbb{P}(V)^{\text{ss}}}\left(\mathcal{O}(-1)\right)$. If it lies on a $\lambda$-fixed component $\mathbb{P}(V_a)$ for $a\neq 0$, it is part of an unstable locus for either $\lambda$ or $\lambda^{-1}$. Thus it lies on $\text{Tot}_{\mathbb{P}(V_0)}\left(\OO(-1)\right)$; the normal bundle is $\bigoplus_{i\in\mathbb{Z}-\{0\}}V_i$, which is symmetric.  
\end{proof}

\begin{prop}
Let $0\leq k\leq n$.
The category $\D_\ell(\X_k)$ is admissible in $\D_\ell\left(\X^{\text{ss}}_{k+1}\right)$.
\end{prop}

\begin{proof}
To simplify the notation, let $\X:=\X_k$, $\Y:=\X_{k+1}$. We will see that the unstable loci are indexed by $(\lambda, a)$, where $\lambda:B\mathbb{G}_m\to\X$ and $a\in\mathbb{Z}$. For each $\Theta$-stratum $\mathcal{S}_{\lambda, a}$ with associated fixed stack $\mathcal{Z}_{\lambda, a}$, choose a real number $w_{\lambda, a}\notin\mathbb{Z}$. 
By \cite[Theorem 3.9]{HL}, there is an admissible subcategory $\mathbb{G}\hookrightarrow D^b(\Y)$ with objects complexes $\F$ such that for any $\Theta$-stratum $\mathcal{S}_{\lambda,a}$:
\[-\frac{n_{\lambda, a}}{2}+ \lambda^*(\ell)+
w_{\lambda, a}\leq \lambda^* \mathcal{F}
\leq \frac{n_{\lambda, a}}{2}+\lambda^*(\ell)+w_{\lambda, a}\]
with the property that
\[\text{res}:\mathbb{G}\cong D^b(\Y^{\text{ss}}).\]
We next characterize $\Theta$-strata and compute $n_{\lambda, a}$.
We can assume that we are in the local case 
\[V/G\xleftarrow{\pi} \text{Bl}_0V/G\cong\text{Tot}_{\mathbb{P}(V)}\left(\mathcal{O}(-1)\right)\hookleftarrow (\text{Bl}_0V)^{\text{ss}}/G,\]
where $G$ is a reductive group, $V$ is a symmetric $G$-representation with $0$ the only $G$-fixed point in $V$, and the linearization is given by the tautological line bundle $\mathcal{O}(1)$.
By the argument of Proposition \ref{indsym}, the unstable loci are determined by pairs 
\[(\lambda, \mathbb{P}(V_a))\] where $V_a$ is the $\lambda$-weight $a$ subspace of $V$ and $a<0$.
The $\lambda$-positive part of the normal bundle is \[N_{\mathbb{P}(V_a)/\mathbb{P}(V)}^{\lambda>0}\cong \bigoplus_{i>a} V_i\]
and $\lambda$ acts with weight $i+a$ on $V_i$. The length of the window $n_{\lambda, a}$
is thus:
\begin{multline}\label{m1}
    n_{\lambda, a}=\langle \lambda, N_{\mathbb{P}(V_a)/\mathbb{P}(V)}^{\lambda>0}\rangle-\langle \lambda, \mathfrak{g}^{\lambda>0}\rangle=\\ \sum_{i>a}(i+a)\text{dim}\,V_i-\langle \lambda, \mathfrak{g}^{\lambda>0}\rangle>
\langle\lambda, V^{\lambda>0}\rangle-
\langle\lambda, \mathfrak{g}^{\lambda>0}\rangle.
\end{multline}

The category $\mathbb{D}_\ell(\X)\subset D^b(\X)$ is defined by the conditions
\begin{equation}\label{m2}
    -\frac{m_\lambda}{2}+\lambda^*(\ell)
    \leq \text{wt}\,\lambda^*\mathcal{F}\leq \frac{m_\lambda}{2}+\lambda^*(\ell),
\end{equation}
where, in the local model $\X=V/G$ from above, $m_\lambda=\langle\lambda, V^{\lambda>0}\rangle-
\langle\lambda, \mathfrak{g}^{\lambda>0}\rangle$.

Similarly, the category $\mathbb{D}_\ell(\Y^{\text{ss}})$ of $D^b(\Y^{\text{ss}})$ is defined by the conditions
\begin{equation}\label{m3}
    -\frac{m_\lambda}{2}+\lambda^*(\ell)
    \leq \text{wt}\,\lambda^*\mathcal{F}\leq \frac{m_\lambda}{2}+\lambda^*(\ell).
\end{equation}
Indeed, a fixed substack $\mathcal{Z}$ associated to a map $\lambda:B\mathbb{G}_m\to\X$ is isomorphic, in the local model, to $\text{Tot}_{\mathbb{P}(V_0)}\left(\OO(-1)\right)$, where $V_0\subset V$ is the $\lambda$-fixed locus. This follows from the analysis at the end of the proof of Proposition \ref{indsym}.

The $\lambda$-fixed loci of $\Y^{\text{ss}}$ are all above $\lambda$-fixed loci of $\X$.
Using \eqref{m1}, \eqref{m2}, \eqref{m3}, we can thus choose weights $w_{\lambda, a}\notin\mathbb{Z}$ such that 
\begin{align*}
\pi^*:D^b(\X)&\to D^b(\Y)\\ \D_\ell(\X)&\subset \D_\ell\left(\Y^{\text{ss}}\right).
\end{align*}
Let $\Phi: D^b(\X)\to \D_\ell(\X)$ be the right adjoint of $\D_\ell(\X)\hookrightarrow D^b(\X)$ which exists by Theorem \ref{symsod}, see Subsection \ref{rightfunctor}. Consider the functor 
\[\Phi\pi_*:\D_\ell\left(\Y^{\text{ss}}\right)\to \D_\ell(\X),\] where recall that $\pi_*$ is the derived functor. We claim that $\pi^*$ and $\Phi\pi_*$ are adjoint. Let $A\in \D_\ell(\X)$ and $B\in \D_\ell\left(\Y^{\text{ss}}\right)$. Then
\[\text{RHom}_{\Y}(\pi^*A, B)\cong\text{RHom}_{\X}(A, \pi_*B)\cong\text{RHom}_{\X}(A, \Phi\pi_*B).\]
Finally, the functor $\pi^*$ is fully faithful. By the projection formula, it suffices to show that $\pi_*\mathcal{O}_{\Y}=\OO_{\X}$, which follows from a direct computation in the local case. Thus $\D_\ell(\X)$ is admissible in $\D_\ell\left(\Y^{\text{ss}}\right)$.
\end{proof}

We thus obtain that:

\begin{cor}\label{kirw}
The category $\D_\ell(\X)$ is admissible in $D^b\left(Y\right)$.
\end{cor}

We finally show that $\mathbb{D}(\X):=\mathbb{D}_0(\X)$ is an NCR of $X$.

\begin{prop}
Let $\ell\in \text{Pic}(\X)_\mathbb{R}$ be such that for any cocharacter $\lambda$, 
\[-\frac{n_\lambda}{2}+\lambda^*(\ell)\leq 0\leq \frac{n_\lambda}{2}+\lambda^*(\ell).\]
Then $\mathbb{D}_\ell(\X)$ is an NCR of $X$.
\end{prop}

\begin{proof}
Consider the inclusion functor and its natural adjoint obtained by Theorem \ref{symsod} and the discussion in Subsection \ref{rightfunctor}:
\begin{align*}
    \iota&: \mathbb{D}_\ell(\X)\hookrightarrow D^b(\X)\\
    \Phi&: D^b(\X)\to \D_\ell(\X).
\end{align*}
Consider the functors \begin{align*}
\pi_*\iota&:\D_\ell(\X)\to D^b(X)\\
\Phi\pi^*&:\text{Perf}\,(X)\to \D_\ell(\X).
\end{align*}
We need to show that for $\F\in \text{Perf}\,(X)$, we have that
\[\pi_*\iota\Phi\pi^*(\F)=\F.\]
The complexes $\pi^*\F$ have weight zero for any cocharacter $\lambda$, and thus they are in $\D_\ell(\X)$. This means that $\iota\Phi\pi^*\F=\pi^*\F$. We have that
\[\pi_*\iota\Phi\pi^*(\F)=\pi_*\pi^*(\F)=\F\otimes \pi_*\mathcal{O}_\X=\F.\]
The last equality follows from $\pi_*\mathcal{O}_\X=\mathcal{O}_X$, see Subsection \ref{goodmoduli}.
Finally, the category $\mathbb{D}_\ell(\X)$ is smooth and proper over $X$ by Corollary \ref{kirw}, and thus it is an NCR of $X$.
\end{proof}

\section{Intersection cohomology for quotient singularities}\label{comp}


In this Section, $\X$ is a symmetric stack satisfying Assumption B. Let \[\pi:\X\to X\] be the good moduli space morphism. 
Denote by $I$ the set of connected components of $\underline{\text{Map}}\,(B\mathbb{G}_m, \X)$, by $o$ the connected component $\X$, and let $J:=I-o$. Further, let $I'$ be the set of connected components of $\underline{\text{Map}}\,(\Theta, \X)$, $o$ the connected component corresponding to $\X$, and let $J':=I'-o$.
For an attracting stack $\mathcal{S}$ in $J'$ with associated fixed stack $\mathcal{Z}$, consider the map
\[p_{\Sa*}q_{\Sa}^*: \HH^\cdot(\Z)\to \HH^\cdot(\X).\]
Define
\[\textbf{B}:=\text{image}\,\left(\bigoplus_{J'} p_{\Sa*}q_{\Sa}^*: \HH^\cdot(\mathcal{Z})\to \HH^\cdot(\X)\right).\]
Recall that $D_{\text{shvs}}(X)$ is the category of complexes of constructible sheaves on a space $X$. Let $\left(\tau^{\leq i}, \tau^{>i}\right)$ be the functors corresponding to the usual t-structure on $D_{\text{shvs}}(X)$ and let $D_{\text{shvs}}^{\leq i}(X)$ the subcategory of sheaves with $\tau^{>i}=0$.
Let $\left({}^p\tau^{\leq i}, {}^p\tau^{>i}\right)$ be the functors associated to the perverse t-structure on $D_{\text{shvs}}(X)$ and denote by ${}^pD_{\text{shvs}}^{\leq i}(X)$ the subcategory of sheaves with ${}^p\tau^{>i}=0$. The map $\pi$ induces a perverse filtration:
\[\textbf{P}^{\leq i}\HH^\cdot(\X):=\text{image}\big(\HH^\cdot\left(X,{}^p\tau^{\leq i}R\pi_*\IC_\X\right)\to \HH^\cdot\left(\X,R\pi_*\IC_\X\right)\big)\subset \HH^\cdot(\X).\]
If $\X$ satisfies Assumption C, then $\textbf{P}^{\leq 0}\HH^\cdot(\X)\cong \IH^\cdot(X)$ by Proposition \ref{prp}.
The main result we prove in this Section is:

\begin{thm}\label{Thmcoh}
Let $\X$ be a symmetric stack satisfying Assumption B.
Then
\[\HH^\cdot(\X)=\textbf{P}^{\leq 0}\HH^\cdot(\X)\oplus \textbf{B}.\]
\end{thm}

We first explain that the BBDG Decomposition Theorem \cite{BBD} implies a Decomposition Theorem for the map $\pi$.

\begin{prop}\label{bbdg}
There is a decomposition 
\[R\pi_*\IC_{\X}\cong\bigoplus_{i\geq 0}{}^p\mathcal{H}^i(R\pi_*\IC_\X)[-i]\]
and each sheaf ${}^p\mathcal{H}^i(R\pi_*\IC_\X)$ is a direct sum of sheaves $\IC_Z(\mathcal{L})$ for $Z\subset X$ and $\mathcal{L}$ a local system on an open smooth subset of $Z$.
\end{prop}

An explicit computation of $R\pi_*\IC_\X$ for $\X$ a stack of representation of a quiver appears in \cite[Theorem 4.6]{MR}, \cite[Proof of Theorem 4.10]{DM}. 

\begin{proof}
The idea, inspired by Totaro's approximations of quotient stacks, is to approximate the stack $\X$ with smooth varieties proper over $X$ and apply the BBDG Decomposition Theorem \cite{BBD}. By Corollary \ref{ahr2}, it suffices to treat the local case $\X=V/G$.
\smallskip

For $n\geq 1$, consider the stacks
\[\X_n:=\left(V\oplus V^{\oplus n}\right)/G\times\mathbb{C}^*,\]
where $\mathbb{C}^*$ acts with weight zero on the first copy of $V$ and with weight $1$ on the summand $V^{\oplus n}$, and $G$ acts naturally on all copies of $V$. Consider the linearization $G\times\mathbb{C}^*\xrightarrow{\text{pr}_2} \mathbb{C}^*$. Define the schemes
\begin{align*}
    S_n&:=\left(V\oplus V^{\oplus n}\right)^{\text{st}}\sslash G\times\mathbb{C}^*\\
    S_n^o&:=\left(V\oplus V^{\oplus n}\right)^{\text{st}, \text{nf}}\sslash G\times\mathbb{C}^*.
\end{align*}
The superscript $\textit{nf}$ means that we take the open subset of the stable locus $\left(V\oplus V^{\oplus n}\right)^{\text{st}}$ not fixed by any elements of $G$. Consider the natural maps:
\begin{equation*}
    \begin{tikzcd}
    \X_n\arrow[d,"t_n"']& S_n^o\arrow[l,"i_n"']\arrow[d,hook,"\ell_n"]\\
    \X\arrow[d,"\pi"']& S_n\arrow[dl,"q_n"]\\
    X.
    \end{tikzcd}
\end{equation*}
Denote by $\pi_n:=\pi t_n:\X_n\to X$. We discuss two preliminary results.

\begin{prop}\label{stable}
We have that $\left(V\oplus V^{\oplus n}\right)^{\text{st}}=\left(V\oplus V^{\oplus n}\right)^{\text{ss}}$.
\end{prop}

\begin{proof}
Let $\mu$ be a cocharacter of $G\times\mathbb{C}^*$. It is enough to show that there are no $\mu$-fixed points on $\left(V\oplus V^{\oplus n}\right)^{\text{ss}}$. Write $\mu=\lambda\cdot z^b$ for $b\in\mathbb{Z}$, $\lambda$ a cocharacter of $G$, and $z$ the identity cocharacter of $\C^*$.

If $b>0$, then $\left(V\oplus V^{\oplus n}\right)^{\mu}$ is unstable. If $b<0$, then $\left(V\oplus V^{\oplus n}\right)^{\mu^{-1}}$ is unstable. If $b=0$, then
\[\left(V\oplus V^{\oplus n}\right)^{\lambda}\subset \left(V\oplus V^{\oplus n}\right)^{\lambda\cdot z\geq 0},\] so $\left(V\oplus V^{\oplus n}\right)^{\lambda}$ is contained in an unstable locus. 
\end{proof}

\begin{prop}\label{truncations}
Fix $a\in\mathbb{Z}$.
For $n$ large enough, we have that
\begin{align*}
    {}^p\tau^{\leq a}R\pi_{n*}\mathbb{Q}_{\X_n}&\cong{}^p\tau^{\leq a}R\pi_{n*}i_{n*}\mathbb{Q}_{S_n^o}\\
    {}^p\tau^{\leq a}Rq_{n*}\mathbb{Q}_{S_n}&\cong{}^p\tau^{\leq a}R\pi_{n*}i_{n*}\mathbb{Q}_{S_n^o}.
\end{align*}
\end{prop}

\begin{proof}
We only explain the first equality; the second one is similar. 
All complexes we consider have cohomology of finite dimension and are bounded on the left. For stacks, functoriality of such complexes is discussed by Laszlo--Olsson \cite{LO1}, \cite{LO2}. Consider the substacks
\[j_n:\mathcal{Z}_n\hookrightarrow \X_n \hookleftarrow S_n^o: i_n.\]
There is a distinguished triangle in $D_{\text{shvs}}(Y)$:
\[j_{n!}j_n^!\mathbb{Q}_{\X_n}=j_{n*}\omega_{\mathcal{Z}_n}[-2\dim\X_n]\to \mathbb{Q}_{\X_n}\to i_{n*}\mathbb{Q}_{S^o_n}\xrightarrow{[1]}.\]

The complex $\omega_{\mathcal{Z}_n}$ is in $D^{\geq -2\dim \mathcal{Z}_n}_{\text{shvs}}(\mathcal{Z}_n)$, see \cite[Section V.2]{Iv}. Let $c_n$ be the codimension of $\mathcal{Z}_n$ in $\X_n$. Then $\omega_{\mathcal{Z}_n}[-2\dim\X_n]\in D^{\geq 2c_n}_{\text{shvs}}(\mathcal{Z}_n).$ Pushforward preserves the categories $D^{\geq \cdot}$, so \[R\pi_{n*}i_{n*}\omega_{\mathcal{Z}_n}[-2\dim\X_n]\in D^{\geq 2c_n}_{\text{shvs}}(X).\] There is a constant $b$ only depending on $X$ such that
\[R\pi_{n*}i_{n*}\omega_{\mathcal{Z}_n}[-2\dim\X_n]\in {}^pD^{\geq b+2c_n}_{\text{shvs}}(X),\] and this implies the desired conclusion for large enough $n$.
\end{proof}

Now we continue the proof of Proposition \ref{bbdg}. By Proposition \ref{stable}, the variety $S_n$ has finite quotient singularities, thus $\IC_{S_n}\cong\mathbb{Q}_{S_n}[\dim S_n]$. By Proposition \ref{truncations}, we have that for $m>n$ large enough,
\[{}^p\tau^{\leq a}R\pi_{n*}\mathbb{Q}_{\X_n}\cong{}^p\tau^{\leq a}Rq_{n*}\mathbb{Q}_{S_n}\cong{}^p\tau^{\leq a}Rq_{m*}\mathbb{Q}_{S_m}.\]
Fix $a\in\mathbb{Z}$. The complex $R\pi_{*}\mathbb{Q}_\X$ is a direct summand of $R\pi_{n*}\mathbb{Q}_{\X_n}$, so ${}^p\tau^{\leq a}R\pi_*\mathbb{Q}_\X$ is a direct summand of ${}^p\tau^{\leq a}R\pi_{n*}\mathbb{Q}_{\X_n}$. 
The Decomposition Theorem for the maps $q_{n}$ implies the desired conclusion. 

\end{proof}

For a cocharacter $\lambda$, define $c_\lambda:=\dim \X-\dim \X^{\lambda\geq 0}$.

\begin{prop}\label{CH}
Assume that $\X=V/G$.
Let $\lambda$ be a cocharacter of $G$ and let $\textbf{B}$ be a semisimple summand of ${}^p\mathcal{H}^i(R\pi_*\IC_\X)$ with support contained in the image of $X^\lambda\to X$. 
Then 
\[\textbf{B}\subset \text{image}\,\left({}^p\mathcal{H}^i\left(R\pi_*p_*\IC_{\X^{ \lambda\geq 0}}[-c_\lambda]\right)\to {}^p\mathcal{H}^i\left(R\pi_*\IC_\X\right)\right).\]
\end{prop}

\begin{proof}
We use the same notations as in the proof of Proposition \ref{bbdg}. First, for any $n$, there is an isomorphism
\begin{equation}
Rt_{n*}\mathbb{Q}_{\X_n}\cong\mathbb{Q}_{\X}\otimes\mathbb{Q}[h],
\end{equation}
where $\mathbb{Q}[h]$ is the polynomial ring in a generator
$h$ of cohomological degree $2$. It suffices to show the statement for summands $\textbf{B}$ of $R\pi_{n*}\mathbb{Q}_{\X_n}$.
Choose $n$ such that $\textbf{B}$ is a summand of ${}^p\mathcal{H}^i\left(Rq_{n*}\mathbb{Q}_{S_n}[\dim X]\right)$. Define
\begin{align*}
    \mathcal{Y}_n&:=\left(V^{\lambda\geq 0}\oplus V^{\oplus n}\right)/G\times\mathbb{C}^*,\\
    W_n&:=\left(V^{\lambda\geq 0}\oplus V^{\oplus n}\right)^{\text{st}}\sslash G\times\mathbb{C}^*\\
    W_n^o&:=\left(V^{\lambda\geq 0}\oplus V^{\oplus n}\right)^{\text{st}, \text{nf}}\sslash G\times\mathbb{C}^*.
\end{align*}
It suffices to show that:
\begin{equation}\label{summand}
\textbf{B}\subset \text{image}\,\left({}^p\mathcal{H}^\cdot\left(R\pi_{n*}p_*\IC_{\mathcal{Y}_n}[-c_\lambda]\right)\to {}^p\mathcal{H}^\cdot\left(R\pi_{n*}\IC_{\X_n}\right)\right).
\end{equation}
Consider the diagram
\begin{equation*}
    \begin{tikzcd}
   S_n\arrow[d,"q_n"']& W_n\arrow[l,"p"']\arrow[d,"r_n"]\\
    X& X^\lambda. \arrow[l]
    \end{tikzcd}
\end{equation*}
By \cite[Proposition 1.5]{CH2}, $\textbf{B}$ appears in the image of
\[{}^p\mathcal{H}^\cdot\left(Rr_{n*}\IC_{W_n}[-c_\lambda]\right)\to 
{}^p\mathcal{H}^\cdot\left(Rq_{n*}\IC_{S_n}\right).\]
Using an argument similar to Proposition \ref{truncations} for $\mathcal{Y}_n$, $W_n$, and $W_n^o$, the claim in \eqref{summand} follows. 
\end{proof}

\begin{prop}\label{prp}
Let $\X=V/G$ be a symmetric stack satisfying Assumption B. Let $\lambda$ be a non-zero cocharacter of $G$. Then 
\begin{align*}
&\text{image}\,\left(R\pi_*p_*\IC_{\X^{ \lambda\geq 0}}[-c_\lambda]\to R\pi_*\IC_\X\right)\subset {}^pD^{\geq 1}_{\text{shvs}}(\X).
\end{align*}
In particular, ${}^p\tau^{\leq 0}R\pi_*\IC_\X$ is a direct sum of IC sheaves with full support.
\end{prop}

\begin{proof}
We show the first statement. We use induction on $\dim G$. 
Consider the diagram 
\begin{equation*}
    \begin{tikzcd}
    \X^{\lambda\geq 0}\arrow[d,"q"']\arrow[dr,"p"]& \\
    \X^\lambda\arrow[d,"\pi^\lambda"'] & \X\arrow[d,"\pi"]\\
    X^\lambda \arrow[r]&X.
    \end{tikzcd}
\end{equation*}
There are natural maps
\[p_*q^*\IC_{\X^\lambda}\to p_*\IC_{\X^{\lambda\geq 0}}[-c_\lambda]\to \IC_\X.\]
The map $q$ is an affine bundle map, so $q_*\mathbb{Q}_{\X^{\lambda\geq 0}}=\mathbb{Q}_{\X^\lambda}$. The stack $\X$ is symmetric, so $\text{reldim}\,q=c_\lambda$. We thus need to show that
\[\text{image}\,\left(R\pi^\lambda_*\IC_{\X^{ \lambda}}\to R\pi_*\IC_\X\right)\subset {}^pD^{\geq 1}_{\text{shvs}}(X).\]
Let $\widetilde{G^\lambda}$ be the quotient of $G^\lambda$ by the torus which acts trivially on $V^\lambda$.
Then \[\widetilde{\pi^\lambda}:\widetilde{\X^\lambda}:=V^\lambda/\widetilde{G^\lambda}\to X^\lambda\] is a good moduli space and $\widetilde{\X^\lambda}$ is a symmetric stack satisfying Assumption B. We thus have that
\begin{align*}
    R\widetilde{\pi^\lambda}_*\IC_{\widetilde{\X^\lambda}}\in {}^pD^{\geq 0}_{\text{shvs}}(\X^\lambda)\\
    R\pi_{\lambda*}\IC_{\X^\lambda}\in {}^pD^{\geq 1}_{\text{shvs}}(\X^\lambda).
\end{align*}
The second statement follows from Proposition \ref{CH} and the fact that $\pi$ is generically finite. 
\end{proof}

We next discuss some preliminary computations regarding the maps $m_\mathcal{S}$. For $\X=A/G$ for $A$ an affine scheme (but one can assume for simplicity that $A$ is an affine space by Corollary \ref{ahr2}) and $\lambda:\C^*\to G$, denote by \[m_\lambda:=p_{\lambda^{-1}*}q_{\lambda^{-1}}^*: \HH^\cdot(\X^\lambda)\to \HH^\cdot(\X).\]
Let $N:=N_{a/A}$ be the normal bundle, and denote by $\textbf{A}$ for the set of weights in $N$, $\textbf{A}_\lambda$, $\textbf{g}_\lambda$ for the set of weights (counted with multiplicities) of $N^{\lambda>0}$, $\mathfrak{g}^{\lambda>0}$ etc. For $\beta$ a weight of $G$, denote by $h_\beta\in \HH^2(T)$. The following computation are standard, for similar computations see \cite[Theorem 2.2]{COHA}, \cite[Proposition 1.2]{YZ}: 

\begin{prop}\label{compcoh}
Let $\lambda$ be a cocharacter of $G$.
Let $x\in \HH^{\cdot}(\X^\lambda)$. Then 
\[m_\lambda(x)=\sum_{w\in W/W^\lambda}w\left((-1)^{|\textbf{A}_\lambda|-|\textbf{g}_\lambda|}x\frac{\prod_{\textbf{A}_{\lambda}}h_\beta}{\prod_{\textbf{g}_{\lambda}}h_\beta}\right).\]
\end{prop}


\begin{prop}\label{Levi}
Let $\lambda$ and $\mu$ be cocharacters of $G$ with the same associated Levi group $L$. Then 
\[\text{image}\,\left(\HH^\cdot(\X^L)\xrightarrow{m_\lambda}\HH^\cdot(\X)\right)=\text{image}\,\left(\HH^\cdot(\X^L)\xrightarrow{m_\mu}\HH^\cdot(\X)\right).\]
\end{prop}

\begin{proof}
Let $y\in \HH^{\cdot}(\X^L)$. By Proposition \ref{compcoh}, we have that:
\begin{align*}
    m_\lambda(y)&=\pm\sum_{w\in W/W^L}w\left(x\frac{\prod_{\textbf{A}_{\lambda}}h_\beta}{\prod_{\textbf{g}_{\lambda}}h_\beta}\right)\\
    m_\mu(y)&=\pm\sum_{w\in W/W^L}w\left(x\frac{\prod_{\textbf{A}_{\mu}}h_\beta}{\prod_{\textbf{g}_{\mu}}h_\beta}\right).
\end{align*}
The representation $N$ is symmetric, so  
\[\{\pm h_\beta|\,\beta\in N^{\lambda>0}\}=\{\pm h_\beta|\,\beta\in N^{\mu>0}\}=\{h_\beta|\, \beta\in N/N^L\},\]
and thus $m_\lambda(y)=\pm m_\mu(y)$.
\end{proof}


Let $\lambda$ be a cocharacter of $G$.
Let $T^{\lambda}\subset G^\lambda$ be the torus which acts trivially on $A^\lambda$, and let $\widetilde{G^\lambda}:=G^\lambda/T^\lambda$. 
Consider the map \[\widetilde{\pi^\lambda}:\widetilde{\X^\lambda}:=A^\lambda/\widetilde{G^\lambda}\to X^\lambda.\]
Define 
\[\HH^\cdot(\X^\lambda)':=\HH^\cdot\left({}^p\tau^{\leq 0}R\widetilde{\pi^\lambda}_*\IC_{\widetilde{\X^\lambda}}\right)\otimes \mathbb{Q}[\mathfrak{t}^\lambda]\hookrightarrow \HH^\cdot(\X^\lambda),\] where the generators of $\mathfrak{t}^\lambda$ have cohomological degree $2$. 
For each Levi group $L$, choose a cocharacter $\lambda_L$ such that $L\cong G^{\lambda_L}$.

\begin{prop}
We have that
\begin{align*}
\textbf{B}&\cong\text{image}\,\left(\bigoplus_L \HH^\cdot\left(\X^{\lambda_L}\right)\to \HH^\cdot(\X)\right)\\
\textbf{B}&\cong\text{image}\,\left(\bigoplus_L \HH^\cdot\left(\X^{\lambda_L}\right)'\to \HH^\cdot(\X)\right).
\end{align*}

\end{prop}

\begin{proof}
The first equality follows from Proposition \ref{Levi} and the second follows by induction and the Decomposition Theorem. 
\end{proof}

\begin{proof}[Proof of Theorem \ref{Thmcoh}]
Let $p\in X$ and recall the setting from Corollary \ref{ahr2}. The restriction of fixed and $\Theta$-stacks to $p^{-1}(\mathcal{U})$ and $\pi^{-1}(U)$ are in a natural bijection. It suffices to prove the statement in the local case $\X=V/G$.

Choose a splitting in the decomposition theorem for $\pi:\X\to X$:
\[\HH^\cdot(\X)\cong\bigoplus_{i\geq 0}{}^p\HH^i(\X),\]
where ${}^p\HH^i(\X):=\HH^\cdot\left({}^p\mathcal{H}^i\left(R\pi_*\IC_\X\right)\right)$. 
For $L$ a Levi subgroup, denote by $\HH^\cdot(\X)_{X^L}$ the cohomology of the summands in $\bigoplus_i {}^p\mathcal{H}^i(R\pi_*\IC_\X)$ with support $X^L\to X$. These are all the supports that appear in the Decomposition Theorem for the map $\pi:\X\to X$. Further, by Propositions \ref{CH}, \ref{prp}, and \ref{Levi},
\[\text{image}\,\left(\HH^\cdot\left(\X^{\lambda_L}\right)'\to \HH^\cdot(\X)\right)=\HH^\cdot(\X)_{X^L}\cong\bigoplus_{i\geq 1}{}^p\HH^i(\X)_{X^L}.\]  
The conclusion follows from Proposition \ref{prp}.
\end{proof}



\section{Categorification of IH for quotient singularities}\label{cat}

In this Section, assume that $\X$ satisfies Assumption B and is symmetric. We will use the category $\mathbb{D}(\X)$ for $\ell$ zero.
For $\X=A/G$ as in Theorem \ref{ahr}, denote by $N=N_{a/A}$ the normal bundle.
We write $\textbf{A}_\lambda$, $\textbf{g}_\lambda$, $\textbf{n}$, $\textbf{n}^\mu$ etc. for the sets of weights (counted with multiplicities) of $N^{\lambda>0}$, $\mathfrak{g}^{\lambda>0}$, $\mathfrak{n}$, $\mathfrak{n}^\mu$ etc. We (abuse notation and) denote by $N^{\lambda>0}$, $\mathfrak{g}^{\lambda>0}$ etc. the sum of weights in $\textbf{A}_\lambda$, $\textbf{g}_\lambda$ etc.
Recall that $\KK_\cdot$ denotes rational K-theory.

\subsection{Notations and definitions}
We begin with some preliminary constructions and definitions. 

\subsubsection{}\label{Weylinv}
There is a natural isomorphism 
\[\KK_\cdot(A/G)\cong \KK_\cdot(A/T)^W.\]

\subsubsection{} 
Let $\chi$ be a weight of $T$. Denote by $\KK_\cdot(A/T)_\chi$ the subspace of $\KK_\cdot(A/T)$ which is the image in K-theory of the inclusion $D^b(A/T)_\chi\subset D^b(A/T)$ of the subcategory generated by locally free $T$-equivariant sheaves with $T$-weight $\chi$. 

\subsubsection{}\label{ag}
Let $\lambda$ be a cocharacter. 
For $I$ a subset of $\textbf{A}_\lambda$, denote by 
\[\sigma_I:=\sum_{\beta\in I}\beta.\]

\subsubsection{}\label{notations}
For two cocharacters $\lambda$ and $\mu$, let $I_\mu^\lambda$ be the set of weights $\beta$ of $\textbf{A}_{\lambda}$ such that $\langle \mu, \beta\rangle<0$; define similarly $J_\mu^\lambda$ for the adjoint representation. We will use the notations:
\begin{align*}
&d^\lambda_\mu=|I^\lambda_\mu|,\\
&e^\lambda_\mu=|J^\lambda_\mu|,\\
&c^\lambda_\mu=|I^\lambda_\mu|-|J^\lambda_\mu|\\
&N^\lambda_\mu=\sum_{I_\mu^\lambda}\beta,\\
&\mathfrak{g}^\lambda_\mu=\sum_{J_\mu^\lambda}\beta,\\
&\mathcal{N}^\lambda_\mu=N^\lambda_\mu-\mathfrak{g}^\lambda_\mu.
\end{align*}


\subsubsection{}\label{S} Let $\lambda$ and $\mu$ be cocharacters of $G$ with associated Levi and Weyl groups $G^\lambda$, $W^\lambda$, $G^\mu$, $W^\mu$. Consider the set $\textbf{S}:=W^\lambda \backslash W\slash W^\mu$. Let $V$ be the set of simple weights. For $s\in \textbf{S}$, consider partitions of $V$ induced by $\lambda$ and $w\mu$:
\begin{equation*}
V=\bigsqcup_{i\in I^\lambda} V_i,\,V=\bigsqcup_{j\in I^{w\mu}} V_j.
\end{equation*}
The sets $I^\lambda$ and $I^{w\mu}$ are the sets of eigenvalues of $\lambda$ and $w\mu$ on $\mathfrak{h}=\mathbb{C}^V$, respectively. They are ordered by the natural ordering of $\mathbb{Z}$. 
Define $V^w_{ij}:=V_i\cap V_j$ and use the lexicographic order on the set $(I^\lambda, I^{w\mu})$. Consider the decomposition 
\begin{equation}\label{V1}
    V=\bigsqcup_{(i, j)\in (I^\lambda, I^{w\mu})} V^w_{ij}.
\end{equation}
Similarly, consider the partitions of $V$ induced by $\mu$ and $w^{-1}\lambda$:
\begin{equation*}
V=\bigsqcup_{i\in I^\mu} V_i,\,V=\bigsqcup_{j\in I^{w^{-1}\lambda}} V_j.\end{equation*}
Define $V'^w_{ij}:=V_i\cap V_j$. Use the lexicographic order on the set $\left(I^\mu, I^{w^{-1}\lambda}\right)$ and consider the decomposition 
\begin{equation}\label{V2}
    V=\bigsqcup_{(i, j)\in (I^\mu, I^{w^{-1}\lambda})} V'^w_{ij}.\end{equation}

Assume next that $\lambda$ and $\mu$ are dominant cocharacters. Let $w\in W^\lambda s W^\mu$ be an element such that the lexicographic order on $(I^\lambda, I^{w_s\mu})$ induces a dominant cocharacter $\nu$; $w$ can be chosen such that it 
does not permute elements of $V_i$ for $i\in I^\lambda$ among themselves, and similarly for $V_j$ for $j\in I^{\mu}$. We can choose $w$ to be the element of minimal length $w_s\in W^\lambda s W^\mu$. 
Let $\nu$ be the dominant cocharacter induced by $\left(I^\lambda, I^{w_s\mu}\right)$. Further, $\left(I^\mu, I^{w_s^{-1}\lambda}\right)$ induces a dominant cocharacter $\nu'$. 

Further, for any $w\in W^\lambda w_s W^\mu$, there exists $w'\in W^\lambda$ such that 
\[\{w'V^w_{ij}| i\in I^\lambda, j\in I^{w\mu}\}=\{V^{w_s}_{ij}| i\in I^\lambda, j\in I^{w_s\mu}\}.\]

\subsubsection{Example.}\label{example} We discuss an example of the construction from the previous Subsection.
Let $G=GL(n)$ and assume that $\lambda$ and $\mu$ are dominant cocharacters with parabolic groups \begin{align*}
    G^{\lambda\geq 0}&=GL(a,b)\\
    G^{\mu\geq 0}&=GL(c,d).
\end{align*} 

We identify the set $V$ with $\{1,\cdots, n\}$. The partition of $V$ corresponding to $\lambda$ is $\{1,\cdots, a\}\sqcup \{a+1,\cdots, n\}$, and the partition of $V$ corresponding to $\mu$ is $\{1,\cdots, c\}\sqcup \{c+1,\cdots, n\}$.
The set $\textbf{S}=\mathfrak{S}_a\times\mathfrak{S}_b\backslash \mathfrak{S}_n\slash \mathfrak{S}_c\times\mathfrak{S}_d$ parametrizes quadruplets $\left(e_1, e_2, e_3, e_4\right)$ such that \begin{align*}
    e_1+e_2&=a,\\ e_3+e_4&=b,\\
    e_1+e_3&=c,\\ e_2+e_4&=d.
\end{align*}
The decomposition \eqref{V1} corresponds to a partition of $V$ in four sets $V_i$ of cardinal $e_i$ for $1\leq i\leq 4$ such that $V_1\sqcup V_2=\{1,\cdots, a\}$; the decomposition \eqref{V2} corresponds to a partition of $V$ in four sets $V'_i$ of cardinals $e_1$, $e_3$, $e_2$, and $e_4$ respectively, such that $V'_1\sqcup V'_2=\{1,\cdots, c\}$. The decomposition corresponding to $\nu$ is 
\[V_1=\{1,\cdots, e_1\}, \cdots, V_4=\{e_1+e_2+e_3+1,\cdots, n\}.\]
The dominant cocharacters $\nu$ and $\nu'$ correspond to the parabolic groups
\begin{align*}
    G^{\nu\geq 0}&=GL(e_1, e_2, e_3, e_4)\\
    G^{\nu'\geq 0}&=GL(e_1, e_3, e_2, e_4).
\end{align*}
The permutation $w_s\in \mathfrak{S}_n$ sends
\begin{align*}
    i&\mapsto i+e_2\text{ for }e_1+1\leq i\leq e_1+e_2,\\
    i&\mapsto i-e_2\text{ for }e_1+e_2+1\leq i\leq e_1+e_2+e_3.
\end{align*}

\subsubsection{}\label{face}

Recall the definition of $\mathbb{W}$ from \eqref{defW}. 
Let $\lambda$ a cocharacter of $G$. Assume that $n_\lambda$ is even and let $b_\lambda=n_\lambda/2$.
Denote by $\mathbb{F}(\lambda)$ the set of weights $\chi$ such that
\[\langle \lambda, \chi\rangle=b_\lambda.\]
Let $\chi$ be a weight satisfying the inequalities in \eqref{bounds} and such that $\chi\in \mathbb{F}(\lambda)$. Then there exists $\psi$ and $w\in W$ such that $w\psi$ is dominant, $w\psi+\rho_L\in \frac{1}{2}\mathbb{W}(\X^L)$, and 
\begin{equation}\label{decompositionchi}
\chi=\frac{1}{2}N^{\lambda>0}-\frac{1}{2}\mathfrak{g}^{\lambda>0}+\psi.
\end{equation}
Here $\rho_L$ is half the sum of positive roots of $\mathfrak{l}=\mathfrak{g}^\lambda$.

Indeed, the condition that $\chi$ satisfies the inequalities in \eqref{bounds} means that there exists $w\in W$ such that $w\chi$ is dominant and 
\[w\chi+\rho\in\frac{1}{2}\mathbb{W}.\]
This also implies that $w\lambda$ is dominant. By \cite[Lemma 3.12]{HLS}, \cite[Corollary 2.4]{P2}, there exists a weight $\tau\in\frac{1}{2}\mathbb{W}(\X^L)$ such that 
\[w\chi+\rho=\frac{1}{2}N^{w\lambda>0}+\tau.\]
Write $\rho=\frac{1}{2}\mathfrak{g}^{w\lambda>0}+\rho_L$. Then there exists $\omega$ dominant such that $\omega+\rho_L\in\frac{1}{2}\mathbb{W}(\X^L)$ and 
\[w\chi=\frac{1}{2}N^{w\lambda>0}-\frac{1}{2}\mathfrak{g}^{w\lambda>0}+\omega.\]
For $\psi=w^{-1}\omega$ we obtain the desired conclusion.

\subsubsection{}\label{cha}   
Recall from the discussion in Subsection \ref{abcd} that the categories $\mathbb{D}(\X)$ may contain complexes supported on attracting stacks. We discuss how to characterize these complexes.

Assume that $\X=A/G$. 
Let $\lambda$ be a cocharacter. Recall that $b_\lambda=n_\lambda/2$. Denote by $\mathbb{D}(\X^\lambda)_b$ the subcategory of $\mathbb{D}(\X^\lambda)$ generated by sheaves of weights $\chi$ such that $\langle \lambda, \chi\rangle=b_\lambda$, see \eqref{decompositionchi} for their description.
A cocharacter $\lambda$ determines a map 
\[m_\lambda:=\frac{1}{|W^\lambda|}p_{{\lambda^{-1}}*}q_{\lambda^{-1}}^*:\KK_\cdot\left(\D(\X^\lambda)_b\right)\to \KK_\cdot\left(\D(\X)\right).\]
To see that the image lies in $\D(\X)$, the sheaves in Proposition \ref{bbw} all have $r$-invariant $\leq 1/2$ by the argument in \cite[Proposition 3.12]{HLS}. 
A cocharacter $\nu:\C^*\to G$ with image in $G^\lambda$  determines a cocharacter of $G^\lambda$ and thus a map
\[m^\lambda_\nu=\frac{1}{|W^\nu|}p_{{\nu^{-1}}*}q_{\nu^{-1}}^*: \KK_\cdot\left(\D(\X^\nu)_b\right)\to \KK_\cdot\left(\D(\X^\lambda)_b\right).\]
There are similarly defined maps in the global case.
\smallskip

A dominant cocharacter $\lambda$ of $G$ determines a restriction map:
\[\Delta_\lambda:=\beta_{\geq b_\lambda}p_\lambda^*: \KK_\cdot\left(\mathbb{D}(\X)\right)\to \KK_\cdot\left(\mathbb{D}(\X^\lambda)_b\right).\] Here, $\beta_{\geq b_\lambda}$ is the functor which considers the top $\lambda$-weight component:
\[\beta_{\geq b_\lambda}: D^b(\X^{\lambda\geq 0})\to D^b(\X^{\lambda\geq 0})_{b_\lambda}\cong D^b(\X^{\lambda})_{b_\lambda},\]
see \cite[Lemma 3.4, Corollary 3.17]{HL2} for a defintion of the functor $\beta_{\geq \cdot}$; the equivalence is induces by the functor $q_\lambda^*$ \cite[Amplification 3.18]{HL2}.
It has the formula from Proposition \ref{computations3}. 
For dominant $\nu$ as above, there is a restriction map:
\[\Delta^\lambda_\nu: \KK_\cdot\left(\mathbb{D}(\X^\lambda)_b\right)\to \KK_\cdot\left(\mathbb{D}(\X^\nu)_b\right).\]

\subsubsection{}\label{sss} Assume that $\X=A/G$. We use the notations and settings of 
Subsections \ref{notations} and \ref{S}. Consider dominant $\lambda$ and $\mu$ inducing dominant $\nu$ and $\nu'$. Then 
\[\{\beta\in N|\,\langle \lambda, \beta\rangle>0, \langle w_s\mu, \beta\rangle<0\}=\{\beta\in N|\,\langle \nu, \beta\rangle>0, \langle \nu', \beta\rangle<0\},\] and so
$c^\lambda_{w_s\mu}=c^\nu_{\nu'},
    \mathcal{N}^\lambda_{w_s\mu}=\mathcal{N}^\nu_{\nu'}.$ A weight $\beta$ of $T$ determines $q_\beta\in K_0(BT)$. 
Define
    \begin{align*}
\widetilde{\text{sw}_s}: \KK_\cdot\left(\mathbb{D}(\X^\nu)_b\right)&\to \KK_\cdot\left(\mathbb{D}(\X^{\nu'})_b\right)\\
y&\mapsto (-1)^{c^\nu_{\nu'}}
w_s^{-1}\left(yq^{-\mathcal{N}^\nu_{\nu'}}\right)=(-1)^{c^\lambda_{w_s\mu}}
w_s^{-1}\left(yq^{-\mathcal{N}^\lambda_{w_s\mu}}\right).
    \end{align*}

\subsection{Computations in K-theory}\label{subcoproduct}
Recall the notations $I, J, I', J'$ from the beginning of Section \ref{comp}. 
For $\mathcal{S}$ an attracting stack in $J$ with fixed locus $\mathcal{Z}$, consider the map
\[m_\mathcal{S}:=p_*q^*: \KK_\cdot(\Z)\to \KK_\cdot(\X)\]
and the subspace of $K_\cdot(\X)$:
\[\textbf{B}\KK_\cdot(\X):=\text{image}\left(\bigoplus_{J'} 
\KK_\cdot(\Z)\to \KK_\cdot(\X)\right).\]
We define $\textbf{P}\KK_\cdot(\X)$ in Subsection \ref{intKtheory}. The main result we prove in this Section is:

\begin{thm}\label{decompK}
There is a decomposition $\KK_\cdot(\X)=\textbf{P}\KK_\cdot(\X)\oplus \textbf{B}\KK_\cdot(\X)$.
\end{thm}

The definition of $\textbf{P}\KK_\cdot(\X)$ is based on the following result. We restrict to the local case $\X=A/G$. The following is the main ingredient in the proof of Theorem \ref{decompK}.

\begin{thm}\label{copr}
Let $\lambda$ and $\mu$ be two dominant cocharacters, let $\textbf{S}$ be the set defined in Subsection \ref{S}, and let $\nu$ be dominant cocharacters as constructed in Subsection \ref{S}.
The following diagram commutes:
\begin{equation*}
\begin{tikzcd}
\KK_\cdot\left(\D(\X^\lambda)_b\right)\arrow[d, "\bigoplus_{\textbf{S}} \Delta^\lambda_\nu"']\arrow[r, "m_\lambda"]& \KK_\cdot(\D(\X))\arrow[d, "\Delta_\mu"]\\
\bigoplus_{\textbf{S}} \KK_\cdot\left(\D(\X^\nu)_b\right) \arrow[r, "\bigoplus_{\textbf{S}} \widetilde{m^\mu_{\nu}}"] & \KK_\cdot\left(\D(\X^\mu)_b\right).
\end{tikzcd}
\end{equation*}
\end{thm}

Before the start of the above result, we list some preliminary computations. The first two are standard, for example, for the first one see \cite[Proposition 1.2]{YZ}, \cite[Propositions 3.1 and 3.2]{P3}: 

\begin{prop}\label{computations}
Consider the maps 
$A/T\xleftarrow{t} A/B\xrightarrow{s} A/G.$
Then the map $s_*t^*:\KK_\cdot(A/T)\to \KK_\cdot(A/G)$ has the formula 
\[s_*t^*(y)=\sum_{w\in W} w\left(\frac{y}{\prod_{\beta\in\textbf{n}}(1-q^{-\beta})}\right).\]
\end{prop}

\begin{prop}\label{computations2}
The map 
$m_{\lambda}: \KK_\cdot(\X^\lambda)\to \KK_\cdot(\X)$ has the formula 
\[m_{\lambda}(x)=\frac{1}{|W^\lambda|}\sum_{w\in W/W^\lambda} 
w\left(x\frac{\prod_{\beta\in \textbf{A}_{\lambda}}(1-q^{-\beta})}
{\prod_{\beta\in \textbf{g}_{\lambda}}(1-q^{-\beta})}\right).\]
\end{prop}

\begin{proof}
Consider the natural maps \begin{align*}
    r&: A\hookrightarrow A^{\lambda\geq 0}\times_{A^{\lambda}}A^{\lambda\leq 0}\to A^\lambda,\\
    v&: A/G^{\lambda\geq 0}\to A/G.
\end{align*}
In the statement of Proposition \ref{bbw}, the sheaves $\mathcal{F}_I$ are $v_*\left(r^*(\mathcal{F})(-\sigma_I)\right)$. The claim thus follows from Proposition \ref{bbw}, see also \cite[Propositions 3.1]{P3} for a similar computation.
\end{proof}

\begin{prop}\label{Leviequal}
Let $\lambda$ and $\mu$ be cocharacters with the same associated Levi group $L\subset G$. 
For $y\in \KK_\cdot(\X^L)$, let $y':=(-1)^{c^\lambda_ \mu}yq^{-\mathcal{N}^{\lambda}_{\mu}}\in \KK_\cdot(\X^L)$. Then \[m_\lambda(y)=m_\mu(y').\]
\end{prop}

\begin{proof}
Using Proposition \ref{computations2}, we have that:
\begin{align*}
    m_\lambda(y)&=\frac{1}{|W^\lambda|}\sum_{W/W^L}
    w\left(y\frac{\prod_{\textbf{A}_{\lambda}}(1-q^{-\beta})}{\prod_{\textbf{g}_{\lambda}}(1-q^{-\beta})}\right)\\
    m_\mu(y)&=\frac{1}{|W^\lambda|}\sum_{W/W^L}w\left(y\frac{\prod_{\textbf{A}_{\mu}}(1-q^{-\beta})}{\prod_{\textbf{g}_{\mu}}(1-q^{-\beta})}\right).
\end{align*}
Further, we have that
\[\frac{\prod_{\textbf{A}_{\mu}}(1-q^{-\beta})}{\prod_{\textbf{g}_{\mu}}(1-q^{-\beta})}=
\frac{\prod_{\textbf{A}_{\lambda}}(1-q^{-\beta})}{\prod_{\textbf{g}_{\lambda}}(1-q^{-\beta})}(-1)^{c^\lambda_\mu}q^{\mathcal{N}^{\lambda}_{\mu}},\]
which implies the desired conclusion.
\end{proof}

\begin{cor}\label{corD}
Under the hypothesis of Proposition \ref{Leviequal}, 
\[\text{image}\,\left(\KK_\cdot\left(\mathbb{D}(\X^\lambda)_b\right)\xrightarrow{m_\lambda}\KK_\cdot(\mathbb{D}(\X))\right)=
\text{image}\,\left(\KK_\cdot\left(\mathbb{D}(\X^\mu)_b\right)\xrightarrow{m_\mu}\KK_\cdot(\mathbb{D}(\X))\right).\]
\end{cor}

\begin{proof}
Assume that $y\in \KK_\cdot\left(\mathbb{D}(\X^\lambda)_b\right)$. Let $\chi$ and $\chi'$ be the weights of $y$ and $y'$, respectively. Then
\[\chi'=\chi+\mathcal{N}^{\mu}_\lambda.\]
By the discussion in Subsection \ref{face} and the decomposition in \eqref{decompositionchi}, there is a weight $\psi$ with the properties mentioned there such that
\begin{align*}
    \chi&=\frac{1}{2}N^{\lambda>0}-\frac{1}{2}\mathfrak{g}^{\lambda>0}+\psi\\
    \chi'&=\frac{1}{2}N^{\mu>0}-\frac{1}{2}\mathfrak{g}^{\mu>0}+\psi,
\end{align*}
and thus $y'=(-1)^{c^\lambda_ \mu}yq^{\mathcal{N}^{\mu}_{\lambda}}\in \KK_\cdot\left(\mathbb{D}(\X^\mu)_b\right)$.
\end{proof}

\begin{prop}\label{computations3}
Let $\mu$ be a dominant character and let $y\in \KK_\cdot(A/T)_\chi$ with $\chi$ on the face $\mathbb{F}(\mu)$. Then the restriction map
$\Delta_\mu: \KK_\cdot(\D(\X))\to \KK_\cdot\left(\D(\X^\mu)_b\right)$ sends 
\[\Delta_\mu\left(\sum_{w\in W} w\left(\frac{y}{\prod_{\beta\in \textbf{n}}(1-q^{-\beta})}\right)\right)=
\sum_{w\in W^\mu} w\left(\frac{y}{\prod_{\beta\in \textbf{n}^\mu}(1-q^{-\beta})}\right).\]
Let $y\in \KK_\cdot(A/T)_\chi$ with $w*\chi$ not on $\mathbb{F}(\mu)$ for any $w\in W$, then
\[\Delta_\mu\left(\sum_{w\in W} w\left(\frac{y}{\prod_{\beta\in \textbf{n}}(1-q^{-\beta})}\right)\right)=0.\]
\end{prop}

\begin{proof}
The weight $\chi$ is dominant up to multiplication by an element of $W^\mu$, so we can assume it is dominant. For $\mathcal{F}\in \D(\X)$, we have that \[\Delta_\mu(\mathcal{F})=\beta_{\geq b_\mu}p_\mu^*(\mathcal{F}).\] For $\mathcal{F}$ with associated representation $\Gamma_G(\chi)$, $\Delta_\mu(\mathcal{F})$ has associated representation $\Gamma_{G^\mu}(\chi)$, and the first part follows. For the second part, by Proposition \ref{Leviequal} we can replace $\chi$ with any $w*\chi$ for $w\in W$, so we can assume that $\chi$ is dominant. Then $\chi$ is not on $\mathbb{F}(\mu)$, so $\langle \mu, p_\mu^*(\mathcal{F})\rangle<b_\mu,$ and the second part thus follows. 
\end{proof}

\begin{prop}\label{computations4}
Consider $\lambda, \tau$ two dominant cocharacters and $\chi$ a dominant weight with $\chi\in \mathbb{F}(\lambda)$. Assume there is a partial sum $\sigma$ of weights in $\textbf{A}_\lambda$  and an element $w\in W$ such that  $w*\left(\chi-\sigma\right)\in\mathbb{F}(\tau)$. Let $\mu=w^{-1}\tau$. Then
\[\sigma=N^\lambda_{\mu}+\sigma',\]
where $\sigma'$ is a partial sum of weights in $(N^{\mu})^{\lambda>0}$. Conversely, any such partial sum $\sigma$ has the property that $w*\left(\chi-\sigma\right)\in\mathbb{F}(\tau)$.
\end{prop}

\begin{proof}
The weight 
\begin{equation}\label{rhalf}
    \left(\chi-\sigma\right)^++\rho
\in \frac{1}{2}\mathbb{W}
\end{equation}
by the same argument as in \cite[Proposition 3.6]{P1}, see also \cite[Proof of Theorem 3.2]{HLS}.
Using the description in Subsection \ref{face}, 
write 
\begin{align*}
    \chi+\rho&=\frac{1}{2}N^{\lambda>0}+\psi\\
    w*\left(\chi-\sigma\right)+\rho&=\frac{1}{2}N^{\tau>0}+\phi',\text{ and so}\\
    \chi-\sigma+\rho&=\frac{1}{2}N^{\mu>0}+\phi,
\end{align*}
where $\psi$ is a sum of weights of $N^{\lambda}$ and $\phi$ is a sum of weights of $N^{\mu}$. For the first two relations above we use that $\lambda$ and $\tau$ are dominant. Then
\begin{equation}\label{sigma}
    \sigma=N^\lambda_\mu+\psi-\phi.
\end{equation}
Write 
\begin{align*}
    \psi&=\psi^{\lambda0}_{\mu+}+\psi^{\lambda0}_{\mu0}+\psi^{\lambda0}_{\mu-}\\
    \phi&=\phi^{\lambda+}_{\mu0}+\phi^{\lambda0}_{\mu0}+\phi^{\lambda-}_{\mu0},
\end{align*}
where $\psi^{\lambda0}_{\mu+}$ is a sum of weights in $(N^\lambda)^{\mu>0}$ etc. Then the decomposition \eqref{sigma} implies that
\begin{align*}
\psi&=\psi^{\lambda0}_{\mu0},\\ \phi&=\phi^{\lambda0}_{\mu0}+\phi^{\lambda-}_{\mu0},\\
\phi^{\lambda0}_{\mu0}&=\psi^{\lambda0}_{\mu0}.
\end{align*}
This implies that $\sigma=N^\lambda_\mu-\phi^{\lambda-}_{\mu0}$, where $-\phi^{\lambda-}_{\mu0}$ is a partial sum of weights in $(N^\mu)^{\lambda>0}$. 

Conversely, if $\sigma=N^\lambda_\mu+\sigma'$, the argument above shows that $\chi-\sigma+\rho=\frac{1}{2}N^{\mu>0}+\phi$ and by \eqref{rhalf} we have that $w*\left(\chi-\sigma\right)\in\mathbb{F}(\tau)$.
\end{proof}

\begin{proof}[Proof of Theorem \ref{copr}]
Let $x\in \KK_\cdot\left(\mathbb{D}(\X^\lambda)_b\right)$. By Proposition \ref{computations}, let $y\in \KK_\cdot(A/T)$ such that
\[x=\sum_{v\in W^\lambda}v\left(\frac{y}{\prod_{\beta\in\textbf{n}^\lambda}(1-q^{-\beta})}\right).\] We may assume that $y$ is of dominant weight $\chi$.
By Proposition \ref{computations2}, we have that
\[p_{{\lambda}^{-1}*}q_{\lambda^{-1}}^*\sum_{v\in W^\lambda}
        v\left(\frac{y}{\prod_{\beta\in\textbf{n}^\lambda}(1-q^{-\beta})}\right)=
        \sum_{u\in W/W^\lambda} u\sum_{I\subset\textbf{A}_\lambda}
        \sum_{v\in W^\lambda}v\left(\frac{(-1)^{|I|}yq^{-\sigma_I}}{\prod_{\beta\in\textbf{n}}(1-q^{-\beta})}\right)
        .\]
The weight of $yq^{-\sigma_I}$ is $\chi-\sigma_I$. By Propositions \ref{computations3} and \ref{computations4}, such an element has non-zero $\mu$-restriction if and only if there exists $w\in W$ such that 
\begin{equation}\label{sigmaI}
    \sigma_I=N^\lambda_{w\mu}+\sigma_{I'},
    \end{equation}
    where $I'$ is a subset of $\textbf{A}^{w\mu}_\lambda$, the set of weights $\beta$ in $\textbf{A}_\lambda$ such that $\langle w\mu, \beta\rangle=0.$ 
    
    Fix $w$. Let $\nu$ and $\nu'$ be the cocharacters constructed as in Subsection \ref{S}. 
The weight $N^\lambda_{w\mu}$ and the set $\textbf{A}^{w\mu}_\lambda$ depend on the coset $W/W^{w\mu}$; if they have associated $\nu$ as above, then they depend on $W^\lambda/W^\nu\subset W/W^{w\mu}$. 
Recall the element $w'$ as the end of Subsection \ref{S}. The element $yq^{-\sigma_I}$ has weight $\chi-\sigma_I$. By \eqref{decompositionchi}, the weight $\left(\chi-\sigma_I\right)^+$ is on $\mathbb{F}(\nu)$. 
Then 
\begin{align*}
    \sum_{v\in W^\lambda}
    v\left(
    \frac{y}{\prod_{\beta\in\textbf{n}^\lambda}
    (1-q^{-\beta})}\right)
    q^{-N^\lambda_{w\mu}-\sigma_{I'}}
    &=\sum_{v\in W^\lambda}
    v\left(
    \frac{y}{\prod_{\beta\in\textbf{n}^\lambda}
    (1-q^{-\beta})}\right)
    w'\left(
    q^{-N^\lambda_{w_s\mu}-\sigma_{J}}
    \right)\\
    &=\sum_{v\in W^\lambda}
    \left(
    \frac{yq^{-N^\nu_{\nu'}-\sigma_{J}}}
    {\prod_{\beta\in\textbf{n}^\lambda}(1-q^{-\beta})}
    \right),
\end{align*}
where $J$ is a subset of $\textbf{A}^{w_s\mu}_{\lambda}$. 
Define
\[m_{\lambda\nu}(x):=\frac{1}{|W^\nu|}\sum_{u\in W/W^\lambda} u\sum_{J\subset\textbf{A}^{w_s\mu}_\lambda}
        \sum_{v\in W^\lambda} v\left(\frac{(-1)^{d^\nu_{\nu'}+|J|}
        yq^{-N^\nu_{\nu'}-\sigma_{J}}}
        {\prod_{\beta\in\textbf{n}}(1-q^{-\beta})}\right).\]
Then $m_\lambda=\sum_{\nu}m_{\lambda\nu}$.  It suffices to show that the following diagram commutes
\begin{equation*}
\begin{tikzcd}
K_\cdot(\D(\X^\lambda)_b)\arrow[d, "\Delta^\lambda_\nu"]\arrow[r, "m_{\lambda\nu}"]& K_\cdot(\D(\X))\arrow[d, "\Delta_\mu"]\\
K_\cdot(\D(\X^\nu)_b) \arrow[r, "\widetilde{m^\mu_{\nu}}"] & K_\cdot(\D(\X^\mu)_b).
\end{tikzcd}
\end{equation*}
Let $\tau$ be the sum of weights $\beta$ in $\mathfrak{n}$ such that $w_s^{-1}\beta$ is not in $\mathfrak{n}$. Then $-w_s\tau=\mathfrak{g}^\lambda_{w_s\mu}$ because the two sides are sums over the weights in the following two sets
\[\{\beta\in \textbf{n}\text{ such that }-\beta\in w_s\textbf{n}\}=\{\beta\in \textbf{g}_\lambda\text{ such that }-\beta\in\textbf{g}_{w_s\mu}\}.\]
Recall the setting of Subsection \ref{sss}. We have that
\begin{align*}
    w_s^{-1}\left(\frac{(-1)^{d^\nu_{\nu'}}
        yq^{-N^\nu_{\nu'}}}
        {\prod_{\beta\in\textbf{n}}(1-q^{-\beta})}\right)
        &=\frac{(-1)^{d^\nu_{\nu'}}
        w_s^{-1}\left(yq^{-N^\nu_{\nu'}}\right)}{(-1)^{e^\nu_{\nu'}}q^{\tau}\prod_{\beta\in\textbf{n}}(1-q^{-\beta})}\\
        &=\frac{(-1)^{c^\nu_{\nu'}}
        w_s^{-1}\left(yq^{-N^\nu_{\nu'}-w_s\tau}\right)}{\prod_{\beta\in\textbf{n}}(1-q^{-\beta})}\\
        &=\frac{(-1)^{c^\nu_{\nu'}}
        w_s^{-1}\left(yq^{-\mathcal{N}^\nu_{\nu'}}\right)}{\prod_{\beta\in\textbf{n}}(1-q^{-\beta})}.
\end{align*}
We can thus rewrite 
\[m_{\lambda\nu}(x):=\frac{1}{|W^\nu|}
\sum_{J'\subset\textbf{A}^{\mu}_{w_s^{-1}\lambda}}
        \sum_{v\in W} 
        v\left(\frac{(-1)^{c^\nu_{\nu'}+|J'|}
        w_s^{-1}\left(yq^{-\mathcal{N}^{\nu'}_{\nu}}\right)q^{-\sigma_{J'}}}
        {\prod_{\beta\in\textbf{n}}(1-q^{-\beta})}\right),\]
and thus 
\[\Delta_\mu m_{\lambda\nu}(x)=\frac{1}{|W^\nu|}\sum_{J'\subset\textbf{A}^{\mu}_{w_s^{-1}\lambda}}
        \sum_{v\in W^\mu} 
        v\left(\frac{(-1)^{c^\nu_{\nu'}+|J'|}
        w_s^{-1}\left(yq^{-\mathcal{N}^{\nu'}_{\nu}}\right)q^{-\sigma_{J'}}}
        {\prod_{\beta\in\textbf{n}^\mu}(1-q^{-\beta})}\right).\]
This is the same as the composition of left-bottom maps:
\begin{align*}
    \widetilde{m^\mu_{\nu}}\Delta^\lambda_\nu(x)
    &=\widetilde{m^\mu_{\nu}}\Delta^\lambda_\nu
    \left(\sum_{v\in W^\lambda}v\left(\frac{y}{\prod_{\beta\in\textbf{n}^\lambda}(1-q^{-\beta})}\right)\right)\\
    &=\widetilde{m^\mu_{\nu}}\sum_{v\in W^\nu}v\left(\frac{y}{\prod_{\beta\in\textbf{n}^\nu}(1-q^{-\beta})}\right)\\
    &=m^\mu_{\nu'}
    \sum_{v\in W^\nu} v\left( \frac{(-1)^{c^\nu_{\nu'}}w_s^{-1}\left(yq^{-\mathcal{N}^\nu_{\nu'}}\right)}{\prod_{\beta\in\textbf{n}^\nu}(1-q^{-\beta}) }
    \right)\\
    &=\frac{1}{|W^\nu|}\sum_{J'\subset\textbf{A}^{\mu}_{w_s^{-1}\lambda}}
        \sum_{v\in W^\mu} 
        v\left(\frac{(-1)^{c^\nu_{\nu'}+|J'|}
        w_s^{-1}\left(yq^{-\mathcal{N}^{\nu'}_{\nu}}\right)q^{-\sigma_{J'}}}
        {\prod_{\beta\in\textbf{n}^\mu}(1-q^{-\beta})}\right)
\end{align*}
where the second equality follows from Proposition \ref{computations3} and the last one by Proposition \ref{computations2}.
\end{proof}


\subsection{Primitive K-theory: the local case}\label{intKtheory}

Recall the setting and notations from the beginning of Subsection \ref{subcoproduct}. In this Subsection, we define $\textbf{P}\KK_\cdot(\X)\subset \KK_\cdot(\mathbb{D}(\X))$ which appears in Theorem \ref{decompK}.
Assume that $\X=A/G$ is a local stack. 


\begin{prop}
Let $\lambda$ be a dominant cocharacter and let $\mathfrak{S}_\lambda\subset W$ be the set of elements $w_s$ for $s\in W^\lambda \backslash W\slash W^\lambda$ such that the weight $\nu$ constructed in Subsection \ref{S} is equal to $\lambda$. Then $\mathfrak{S}_\lambda$ is a group and it acts on $\KK_\cdot(\D(\X^\lambda)_b)$ via $\widetilde{\text{sw}}$.
\end{prop}

\begin{proof}
The elements $w_s$ above are the elements of $W$ that induce permutations of $I^\lambda$ and which do not permute elements in $V_i$ for $i\in I^\lambda$ among themselves. These elements are clearly closed under multiplication and taking inverses. 

For the second part, consider elements $w_1, w_2\in \mathfrak{S}_\lambda$ and let $y\in \KK_\cdot\left(\D(\X^\lambda)_b\right)$. Let $\widetilde{\text{sw}_1}$, $\widetilde{\text{sw}_2}$, $\widetilde{\text{sw}_3}$ be the swap maps for the elements $w_1, w_2$, and $w_2w_1$, respectively. We need to show that $\widetilde{\text{sw}_1}\widetilde{\text{sw}_2}=\widetilde{\text{sw}_3}$:
\begin{align*}
    &(-1)^{c^\lambda_{w_1\lambda}}
    w_1^{-1}
    \left((-1)^{c^\lambda_{w_2\lambda}}w_2^{-1}\left(
    yq^{-\mathcal{N}^\lambda_{w_2\lambda}}
    \right)q^{-\mathcal{N}^\lambda_{w_1\lambda}}
    \right)=\\
    &(-1)^{c^\lambda_{w_1\lambda}+ c^\lambda_{w_2\lambda}}
    w_1^{-1}w_2^{-1}\left(yq^{-\mathcal{N}^\lambda_{w_2\lambda}-\mathcal{N}^{w_2\lambda}_{w_2w_1\lambda}}\right)=\\
    &(-1)^{c^\lambda_{w_2w_1\lambda}}
    (w_2w_1)^{-1}\left(
    yq^{-\mathcal{N}^\lambda_{w_2w_1\lambda}}
    \right).
\end{align*}
\end{proof}
In the example from Subsection \eqref{example}, $\mathfrak{S}_\lambda$ is trivial unless $a=b$, case in which it is the symmetric group $\mathfrak{S}_2$. More generally, for $\lambda$ a cocharacter of $GL(n)$ with corresponding decomposition in distinct parts $d_1, \cdots, d_k$ with multiplicities $m_1, \cdots, m_k$, the group $\mathfrak{S}_\lambda$ is the product of symmetric groups $\times_{i=1}^k\mathfrak{S}_{m_i}$.
In the framework of the above Proposition, denote by 
\begin{align*}
\text{Sym}_\lambda:\KK_\cdot\left(\D(\X^\lambda)_b\right)&\to \KK_\cdot\left(\D(\X^\lambda)_b\right)\\
x&\mapsto \sum_{\sigma\in \mathfrak{S}_\lambda}\sigma(x).
\end{align*}
Define \[\textbf{B}\KK_\cdot\left(\D(\X^\lambda)_b\right)=\text{image}\left(\bigoplus_{\nu}m_\nu: \KK_\cdot\left(\D(\X^\nu)_b\right)\to \KK_\cdot\left(\D(\X^\lambda)_b\right)\right),\] where the sum is over all non-trivial cocharacters $\nu$ of $G^\lambda$.
We define $\textbf{P}\KK_\cdot\left(\D(\X^\lambda)_b\right)$ inductively on $\dim G^\lambda$  
such that 
\begin{equation}\label{OPB}
    \KK_\cdot\left(\D(\X^\lambda)_b\right)=\textbf{P}\KK_\cdot\left(\D(\X^\lambda)_b\right)\oplus \textbf{B}\KK_\cdot\left(\D(\X^\lambda)_b\right).\end{equation}
There are then natural surjections $\pi_\lambda: \KK_\cdot\left(\D(\X^\lambda)_b\right)\twoheadrightarrow \textbf{P}\KK_\cdot\left(\D(\X^\lambda)_b\right)$.

When $\dim G^\lambda=0$, then $\textbf{P}\KK_\cdot\left(\D(\X^\lambda)_b\right)= \KK_\cdot\left(\D(\X^\lambda)_b\right)$. Assume that $\dim G>0$. For any Levi $L<G$, choose a dominant cocharacter $\lambda_L$ such that $G^{\lambda_L}=L$. Denote by $m_H=m_{\lambda_H}$ etc.
Define
\[\textbf{P}\KK_\cdot\left(\D(\X^{\lambda_L})_b\right)=
\bigcap_{H<L}\left(\text{Sym}_H \pi_H \Delta_H: \KK_\cdot\left(\mathbb{D}(\X^{\lambda_L})_b\right)
\to
\textbf{P}\KK_\cdot\left(\mathbb{D}(\X^{\lambda_H})_b\right)\right).\]

Denote by $\iota: \textbf{P}\KK_\cdot\left(\D(\X^{\lambda_L})_b\right)\hookrightarrow \KK_\cdot\left(\D(\X^{\lambda_L})_b\right)$ the natural map, and denote by $\Phi_H:=\text{Sym}_H \pi_H \Delta_H$.

\begin{prop}\label{compequal}
The composition 
\[\textbf{P}\KK_\cdot\left(\D(\X^{\lambda_L})_b\right)
\xrightarrow{m_L}
\KK_\cdot(\D(\X))\xrightarrow{\pi_L\Delta_L}
\textbf{P}\KK_\cdot\left(\D(\X^{\lambda_L})_b\right)\] is $\pi_L\Delta_L m_L(x)=\frac{1}{|W^\lambda|} \sum_{\sigma\in\mathfrak{S}_\lambda}\sigma(x)$.
\end{prop}

\begin{proof}
By Theorem \ref{copr}, we have that
\[\Delta_\lambda m_\lambda(x)=\sum_{\textbf{S}}\widetilde{m_{\nu'}}\Delta_\nu(x).\] For $\nu'$ different from $\lambda$, the element $\widetilde{m_{\nu'}}\Delta_\nu(x)$ is in $\textbf{B}\KK_\cdot\left(\D(\X^{\lambda_L})_b\right)$. Then 
\[\pi_L\Delta_L m_L(x)=
\sum_{\mathfrak{S}_\lambda}\widetilde{m_{\nu'}}\Delta_\nu(x)=\frac{1}{|W^\lambda|}\sum_{\mathfrak{S}_\lambda}\sigma(x).\]
\end{proof}

\begin{prop}\label{compdiff}
Let $L$ and $E$ be  proper Levi groups of $G$ such that $E\nsubseteq L$. Let $\lambda$ and $\mu$ be the associated cocharacters to these Levi groups. 
The composition 
\[\textbf{P}\KK_\cdot\left(\D(\X^\lambda)_b\right)\xrightarrow{m_L}
\KK_\cdot(\D(\X))\xrightarrow{\Delta_E}\KK_\cdot\left(\D(\X^\mu)_b\right)\xrightarrow{\pi_E} \textbf{P}\KK_\cdot\left(\D(\X^\mu)_b\right)\] is zero.
\end{prop}

\begin{proof}
By Theorem \ref{copr}, we have that
\[\Delta_E m_L(x)=\sum_{\textbf{S}}\widetilde{m_{\nu'}}\Delta_\nu(x).\] 
If $E\nsubseteq L$, there is no $s\in \textbf{S}$ such that $\nu'=\mu$, and so the right hand side is in $\textbf{B}\KK_\cdot(\D(\X^\mu)_b)$. 
\end{proof}

We now assume the statements in \eqref{OPB} for $L<G$.

\begin{prop}
There is a surjection 
\[\bigoplus_{L<G}
\textbf{P}\KK_\cdot\left(\mathbb{D}(\mathcal{X}^{\lambda_L})_b\right)^{\mathfrak{S}_L}\twoheadrightarrow \textbf{B}\KK_\cdot(\D(\X)).\]
\end{prop}

\begin{proof}
The image of $m_{L}: \textbf{P}\KK_\cdot\left(\D(\X^{\lambda_L})_b\right)\to \KK_\cdot(\D(\X))$ factors through the symmetrization map
\[\textbf{P}\KK_\cdot\left(\D(\X^{\lambda_L})_b\right)\xrightarrow{\text{Sym}_L}
\textbf{P}\KK_\cdot\left(\D(\X^{\lambda_L})_b\right)^{\mathfrak{S}_L}\xrightarrow{m_L}
\KK_\cdot(\D(\X))\] by Proposition \ref{Leviequal}. 
The statement follows using \eqref{OPB} for $L<G$
and Proposition \ref{corD}. 
\end{proof}


\begin{thm}\label{deco}
There is a decomposition 
\[\left(\iota, \bigoplus_{L<G}m_{\lambda_L}\right):
\textbf{P}\KK_\cdot(\mathbb{D}(\X))\oplus\bigoplus_{L<G} \textbf{P}\KK_\cdot\left(\mathbb{D}(\mathcal{X}^{\lambda_L})_b\right)^{\mathfrak{S}_L}\xrightarrow{\sim}
\KK_\cdot(\mathbb{D}(\X)).\]
\end{thm}

\begin{proof}
Using Propositions \ref{compequal} and \ref{compdiff}, we see that the map is an injection. To see it is a surjection, let $x\in \KK_\cdot(\D(\X))$ and assume that $\Phi_H(x)=0$ for all $L<H<G$ and 
$\pi_L(x)\neq 0$. By Propositions \ref{compequal} and \ref{compdiff}, there is a constant $c$ such that
\[y:=x-c\cdot m_L\Phi_L(x)\in \KK_\cdot(\D(\X))\]
satisfies $\Phi_H(y)=0$ for $L\leq H<G$. Repeating this process, we see that the map is indeed surjective.
\end{proof}

We now prove Theorem \ref{decompK} in the local case.

\begin{cor}\label{thm5}
There is a decomposition $\KK_\cdot(\X)=\textbf{P}\KK_\cdot(\X)\oplus \textbf{B}\KK_\cdot(\X)$.
\end{cor}

\begin{proof}
This follows from Theorem \ref{symsod}, Proposition \ref{Leviequal}, and Theorem \ref{deco}.
\end{proof}

\subsection{Compatibility of decompositions along \'etale maps}\label{compa}
Let $e:\X'\to\mathcal{X}$ be an \'etale map. Let $\mathcal{S}$ be a $\Theta$-stack of $\X$ with associated fixed stack $\Z$. Let $\mathcal{S}'$ be a $\Theta$-stack in $\X'$ contained in $e^{-1}(\mathcal{S})$, and let $\Z'$ be its associated fixed stack. Finally, let $w\in \mathbb{Z}$. Then \begin{align*}
    e^*: D^b(\Z)_w\to D^b(\Z')_w,\, e^*: D^b(\mathcal{S})\to D^b(\mathcal{S}'),\\
    e_*: D^b(\Z')_w\to D^b(\Z)_w,\, e^*: D^b(\mathcal{S}')\to D^b(\mathcal{S}).
\end{align*}
By the construction of the categories $\mathbb{D}$ from Section \ref{ncr}, we obtain functors \begin{align*}
    e^*: \mathbb{D}(\X)&\to \mathbb{D}(\Y),\\
    e_*: \mathbb{D}(\Y)&\to \mathbb{D}(\X).
\end{align*}
By the construction of the spaces $\textbf{P}\KK$ and $\textbf{B}\KK$, we see that $e_*$ and $e^*$ respect these spaces for $\X$ and $\Y$ quotient stacks of smooth affine varieties by reductive groups.

\subsection{Primitive K-theory: the global case}

Let $\X$ be a symmetric stack satisfying Assumption B and let $X$ be its good moduli space. Consider a direct system of \'etale covers $\mathcal{A}$ containing \'etale maps $\Y\to\X$ as in Theorem \ref{ahr}. 
Then, by \cite[Corollary 2.17]{Th}:
\begin{equation}\label{thomason}
\check{\HH}^p\left(\mathcal{A}, \KK_q(-)\right)\Rightarrow \KK_{q-p}(\X),
\end{equation}
where $\check{\HH}$ denotes \v{C}ech cohomology and the spectral sequence converges strongly. It is essential that we use rational K-theory to obtain this statement.
Any $\lambda:B\mathbb{G}_m\to\X$ induces a cocharacter $\lambda$ in local charts $\Y$. Further, any local attracting locus corresponds to a map $\lambda:B\mathbb{G}_m\to\X$ and thus determines an attracting stack. Denote by $\X^\lambda$ the corresponding fixed stack.
For $\lambda: B\mathbb{G}_m\to \X$, define
$\KK_\cdot^\lambda(\Y):=\KK_\cdot\left(\Y^\lambda\right)$. Then
\[\check{\HH}^p\left(\mathcal{A}, \KK_q^\lambda(-)\right)\Rightarrow \KK_{q-p}\left(\X^\lambda\right).\]
Thus the following spectral sequence converges strongly:
\begin{equation}\label{bthomason}
\check{\HH}^p\left(\mathcal{A}, \textbf{B}\KK_q(-)\right)\Rightarrow \textbf{B}\KK_{q-p}(\X).
\end{equation}
By Theorem \ref{thm5}, \eqref{thomason}, and \eqref{bthomason},
we thus obtain that the following spectral sequence converges strongly and define $\textbf{P}\KK_\cdot(\X)$ such that
\[\check{\HH}^p\left(\mathcal{A}, \textbf{P}\KK_q(-)\right)\Rightarrow \textbf{P}\KK_{q-p}(\X).\]

\begin{proof}[Proof of Theorem \ref{decompK}]
The decomposition claimed in 
Theorem \ref{decompK} follows from the construction of $\textbf{P}\KK_\cdot(\X)$ and $\textbf{B}\KK_\cdot(\X)$ and by Theorem \ref{thm5}.
\end{proof}

\subsection{Categorification of intersection cohomology}
\subsubsection{}
The categories $D^b(\X)$ have natural dg enhancements and the admissible subcategories $\mathbb{D}(\X)$ also have natural dg enhancements \cite[Section 4.1]{}.
Recall the definitions from Subsection \ref{ncmotives}.
The splitting \[\KK_0(\mathbb{D}(\X))\to \textbf{P}\KK_0(\mathbb{D}(\X))\to \KK_0(\mathbb{D}(\X))\] induces an idempotent of $e_\mathcal{\X}\in \text{Hom}_{\textbf{Hmo}_{0;\mathbb{Q}}}(\mathbb{D}(\X), \mathbb{D}(\X))$. Indeed, all the functors used to construct the above splitting are constructed from the functors for attracting and fixed stacks
\begin{align*}
    p_*&: D^b(\mathcal{S})\to D^b(\X)\\
    p^*&: D^b(\X)\to D^b(\mathcal{S})\\
    q^*_w&: D^b(\mathcal{Z})_w\xrightarrow{\sim} D^b(\mathcal{S})_w,
\end{align*}
see \cite[Amplification 3.18]{HL2} for the last functor. Thus the functors used
are induced by Fourier--Mukai transforms in $\text{rep}\left(D^b(\mathcal{S}), D^b(\X)\right)$ and $\text{rep}\left(D^b(\mathcal{Z}), D^b(\X)\right)$, respectively. 
We thus obtain a noncommutative motive:
\[\mathbb{D}^{\text{nc}}(\X):=
\left(\mathbb{D}(\X),e_\X\right)\in \textbf{Hmo}_{0;\mathbb{Q}}^{\natural}.\]

Consider the Chern character
\[\text{ch}: \KK_\cdot(\X)\to \widehat{\HH^{\cdot}}(\X):= \prod_{j\in \mathbb{Z}}\HH^{i+2j}(\X).\]
We write $\text{gr}^\cdot K_\cdot(\X)$ for the associated graded with respect to the codimension filtration \cite[Definition 3.7, Section 5.4]{G}. By Theorems \ref{Thmcoh} and \ref{decompK}, we obtain:

\begin{cor}
There is an inclusion
$\text{gr}^\cdot \KK_\cdot(\mathbb{D}^{\text{nc}}(\X))\subset \textbf{P}^{\leq 0}\HH^\cdot(\X).$
\end{cor}

If $\X$ is symmetric and satisfies Assumption C, then $\textbf{P}^{\leq 0}\HH^\cdot(\X)\cong\IH^\cdot(X)$. In this situation, we define \textit{the intersection K-theory of} $X$:
\[\IK_\cdot(X):=\KK_\cdot(\mathbb{D}^{\text{nc}}(\X)).\]
There is a natural map \[\text{ch}:\IK_\cdot(X)\to \IH^\cdot(X)\] obtained using the splittings from Theorem \ref{Thmcoh} and Theorem \ref{decompK}:
\[\IK_\cdot(X)\hookrightarrow \KK_\cdot(\X)\xrightarrow{\text{ch}} \widehat{\HH^\cdot}(\X)\twoheadrightarrow \IH^\cdot(X).\]
Directly from the definition of $\mathbb{D}^{\text{nc}}(\X)$, intersection K-theory satisfies a version of Kirwan surjectivity \cite[Theorem 2.5]{Ki}:
\[\KK_\cdot(\X)_\mathbb{Q}\twoheadrightarrow \IK_\cdot(X).\]

\subsubsection{}\label{Kirwa}
Denote by $\KK^{\text{top}}$ the Blanc topological K-theory of a dg category \cite{Bl}. Recall that for $\X$ a smooth stack, $\KK^{\text{top}}\left(D^b(\X)\right)$ recovers the Atiyah-Segal equivariant topological K-theory \cite[Theorem 3.9]{HLP}. 
For $j\in \{0, 1\}$, denote by $\text{gr}^\cdot\KK_j^{\text{top}}(\X)$ the associated graded with respect to filtration $F^{\geq i}$ induced by $\HH^{\geq j+2i}(\X)$ via \[\text{ch}: \KK_j^{\text{top}}(\X)\to \prod_{i\in\mathbb{Z}}\HH^{j+2i}(\X).\] The idempotent $e$ induces a well-defined direct summand $\KK^{\text{top}}\left(\mathbb{D}^{\text{nc}}(\X)\right)$ of $\KK^{\text{top}}\left(\mathbb{D}(\X)\right)$. For $j\in \{0, 1\}$, we have that 
\[\text{gr}^i \KK^{\text{top}}_j(\X)\cong \HH^{j+2i}(\X,\mathbb{Q}).\]
Indeed, it suffices to show the statement for quotient stacks, case in which both sides can be computed using Totaro's approximations $S_n^o$ of $\X$ from the proof of Proposition \ref{Thmcoh}. 
The isomorphism for a scheme $S_n^o$ holds by the Atiyah-Hirzebruch theorem. By Theorems \ref{Thmcoh} and \ref{decompK}, we have that $\text{gr}^i \KK^{\text{top}}_j\left(\mathbb{D}^{\text{nc}}(\X)\right)\cong \IH^{j+2i}(X,\mathbb{Q})$ for $j\in \{0, 1\}$. If the natural map
\begin{equation}\label{iso123}
    \KK^{\text{top}}_j\left(\mathbb{D}^{\text{nc}}(\X)\right)\xrightarrow{\sim} \text{gr}^i \KK^{\text{top}}_j\left(\mathbb{D}^{\text{nc}}(\X)\right),\end{equation}
is an isomorphism, then \begin{equation}\label{kirwaniso}
    \KK^{\text{top}}_j\left(\mathbb{D}^{\text{nc}}(\X)\right)\cong \IH^{j+2i}(X,\mathbb{Q})\end{equation}
    for $j\in \{0, 1\}$. We claim that if the Kirwan resolution $Y\to X$ is a scheme, then \eqref{iso123} holds. Indeed, it suffices to check \eqref{iso123} for $\mathbb{D}(\X)$ instead of $\mathbb{D}^{\text{nc}}(\X)$. It suffices to check that the Chern character map for $\mathbb{D}(\X)$ is injective. By Corollary \ref{kirw}, it suffices to check that the Chern character map for $Y$ is injective. For a scheme $Y$, the Chern character is an isomorphism by the Atiyah-Hirzebruch theorem. 

Denote by $\HP$ the periodic cyclic homology of a dg category. 
The idempotent $e$ induces a well-defined direct summand $\HP_\cdot\left(\mathbb{D}^{\text{nc}}(\X)\right)$ of $\HP_\cdot\left(\mathbb{D}(\X)\right)$. 
By \cite[Theorem A]{HLP} and \eqref{kirwaniso}, we also obtain that if the Kirwan resolution $Y\to X$ is a scheme, then 
\[\HP_i\left(\mathbb{D}^{\text{nc}}(\X)\right)\cong \bigoplus_{j\in\mathbb{Z}} \IH^{i+2j}(X,\C).\]


\begin{thebibliography}{30}

\bibitem{A}
J.~Alper.
\newblock{Good moduli spaces for Artin stacks}.
\newblock{\em Ann. Inst. Fourier (Grenoble)} 63 (2013), no. 6, 2349–-2402.

\bibitem{AHR}
J.~Alper, J.~Hall, D.~Rydh.
\newblock{A Luna étale slice theorem for algebraic stacks}.
\newblock {\em Ann. of Math.} (2) 191 (2020), no. 3, 675–-738.

\bibitem{AHLH}
J.~Alper, D.~Halpern-Leistner, J.~Heinloth.
\newblock{Existence of moduli spaces for algebraic stacks}
\newblock \url{http://arxiv.org/abs/1812.01128}, 2018.


\bibitem{BFK}
M.~Ballard, D.~Favero, L.~Katzarkov.
\newblock{ Variation of geometric invariant theory quotients and derived categories}.
\newblock{\em J. Reine Angew. Math.} 746 (2019), 235–-303. 




\bibitem{BBD}
A.A.~Beĭlinson, J.~Bernstein, P.~Deligne. \newblock{Faisceaux pervers.}
\newblock{\em Analysis and topology on singular spaces, I} (Luminy, 1981), 5--171, Astérisque, 100, Soc. Math. France, Paris, 1982.



\bibitem{Bl}
A.~Blanc.
\newblock{Topological K-theory of complex noncommutative spaces}.
\newblock{\em Compos. Math.} 152 (2016), no. 3, 489--555.

\bibitem{BO}
A.~Bondal, D.~Orlov.
\newblock{Derived categories of coherent sheaves}.
\newblock {\em Proceedings of the International Congress of Mathematicians}, Vol. II (Beijing, 2002), 47–-56, Higher Ed. Press, Beijing, 2002.


\bibitem{BFN}
A.~Braverman, M.~Finkelberg, H.~Nakajima.
\newblock{Instanton moduli spaces and $\mathcal{W}$-algebras.}
\newblock {\em Astérisque} No. 385 (2016), vii+128 pp.





\bibitem{dCM}
M. A.~de Cataldo, L.~Migliorini.
\newblock{The Chow motive of semismall resolutions.}
\newblock{\em Math. Res. Lett.} 11 (2004), no. 2-3, 151--170.



\bibitem{dCM2}
M. A.~de Cataldo, L.~Migliorini.
\newblock{The decomposition theorem, perverse sheaves and the topology of algebraic maps.} 
\newblock {\em Bull. Amer. Math. Soc. (N.S.)} 46 (2009), no. 4, 535--633.

\bibitem{C}
S.~Cautis.
\newblock{Clasp technology to knot homology via the affine Grassmannian}.
\newblock{\em Math. Ann.} 363 (2015), no. 3-4, 1053--1115. 


\bibitem{CK}
S.~Cautis, J.~Kamnitzer.
\newblock{Quantum K-theoretic geometric Satake: the $SL_n$ case.}
\newblock{\em Compos. Math.} 154 (2018), no. 2, 275--327.



\bibitem{CH1}
A.~Corti, M.~Hanamura.
\newblock{Motivic decomposition and intersection Chow groups. I.}
\newblock{\em Duke Math. J.} 103 (2000), no. 3, 459--522.


\bibitem{CH2}
A.~Corti, M.~Hanamura.
\newblock{Motivic decomposition and intersection Chow groups. II.}
\newblock {\em Pure Appl. Math. Q.} 3 (2007), no. 1, Special Issue: In honor of Robert D. MacPherson. Part 3, 181--203.


\bibitem{DM}
B.~Davison, S.~Meinhardt.
\newblock{ Cohomological Donaldson-Thomas theory of a quiver with potential and quantum enveloping algebras}.
\newblock{\em Invent. Math.} 221 (2020), no. 3, 777--871. 



\bibitem{E}
J.~Eberhardt.
\newblock{K-motives and Koszul Duality.}
\newblock \url{https://arxiv.org/pdf/1909.11151.pdf}, 2019.

\bibitem{ER}
D.~Edidin, D.~Rydh.
\newblock{Canonical reduction of stabilizers for Artin stacks with good moduli spaces}.
\newblock {\em Duke Mathematical Journal}, Volume 170, pp 827--880.

\url{https://arxiv.org/abs/1710.03220}.





\bibitem{G}
H.~Gillet.
\newblock{K-theory and intersection theory.}
\newblock{\em Handbook of K-theory.} Vol. 1, 2, 235--293, Springer, Berlin, 2005.





\bibitem{HL2}
D.~Halpern-Leistner.
\newblock{ The derived category of a GIT quotient}.
\newblock{\em
Journal of the American Mathematical Society} 28 (3), 871--912.



\bibitem{HL}
D.~Halpern-Leistner.
\newblock{Derived $\Theta$-stratifications and the D-equivalence conjecture}.
\newblock{
\url{https://arxiv.org/pdf/2010.01127v1.pdf}.}




\bibitem{HLS}
D.~Halpern-Leistner, S.~Sam.
\newblock{Combinatorial constructions of derived equivalences}.
\newblock{\em Journal of the American Mathematical Society} 33 (3), 735--773.


\bibitem{HLP}
D.~Halpern-Leistner and D.~Pomerleano.
\newblock{Equivariant Hodge theory and noncommutative geometry.} \newblock{\em Geom. Topol.} 24 (2020), no. 5, 2361--2433.


\bibitem{Iv}
B.~Iversen.
\newblock{Cohomology of sheaves}.
\newblock{Universitext. Springer–Verlag, Berlin, 1986}.


\bibitem{LO1}
Y.~Laszlo, M.~Olsson.
\newblock{The six operations for sheaves on Artin stacks. I. Finite coefficients.}
\newblock{\em Publ. Math. Inst. Hautes Études Sci.} No. 107 (2008), 109--168.

\bibitem{LO2}
Y.~Laszlo, M.~Olsson.
\newblock{The six operations for sheaves on Artin stacks. II. Adic coefficients.}
\newblock{\em Publ. Math. Inst. Hautes Études Sci.} No. 107 (2008), 169--210.




\bibitem{Ki}
F.~Kirwan.
\newblock{Rational intersection cohomology of quotient varieties}. 
\newblock{\em Invent. Math.} 86 (1986), no. 3, 471--505.

\bibitem{Ki2}
F.~Kirwan.
\newblock{Rational intersection cohomology of quotient varieties. II}.
\newblock{\em Invent. Math.} 90 (1987), no. 1, 153--167.



\bibitem{COHA}
M.~Kontsevich, Y.~Soibelman.
\newblock {Cohomological Hall algebra, exponential Hodge structures and motivic
  Donaldson--Thomas invariants}.
\newblock {\em {Comm. Num. Th. and Phys.}}, 5(2):231--252, 2011.



\bibitem{K1}
A.~Kuznetsov.
\newblock{
Lefschetz decompositions and categorical resolutions of singularities}. \newblock{\em
Selecta Math.} (N.S.) 13 (2008), no. 4, 661–-696. 



\bibitem{K3}
A.~Kuznetsov.
\newblock{Semiorthogonal decompositions in algebraic geometry.}
\newblock{\em Proceedings of the International Congress of Mathematicians}—Seoul 2014. Vol. II, 635--660.

\bibitem{K2}
A.~Kuznetsov.
\newblock{Hochschild homology and semiorthogonal decompositions}
\newblock \url{https://arxiv.org/pdf/0904.4330.pdf}.






\bibitem{L}
G.~Lusztig.
\newblock{Intersection cohomology methods in representation theory}.
\newblock{A plenary address presented at the International Congress of Mathematicians held in Kyoto, August 1990}.

\bibitem{MR}
S.~Meinhardt, M.~Reineke.
\newblock{Donaldson-Thomas invariants versus intersection cohomology of quiver moduli}.
\newblock{\em J. Reine Angew. Math.} 754 (2019), 143--178. 





\bibitem{P1}
T.~P\u{a}durariu.
\newblock{Intersection K-theory.}
\newblock \url{https://arxiv.org/abs/2103.06223}.




\bibitem{P2}
T.~P\u{a}durariu.
\newblock{Categorical and K-theoretic Hall algebras for quivers with potential.}
\newblock  \url{https://arxiv.org/abs/2107.13642}, to appear in {\em J. Inst. Math. Jussieu}.



\bibitem{P3}
T.~P\u{a}durariu.
\newblock{K-theoretic Hall algebras of quivers with potential as Hopf algebras.}
\newblock  \url{https://arxiv.org/abs/2106.05169}.


\bibitem{SvdB1}
\v{S}.~\v{S}penko, M.~Van den Bergh.
\newblock{ Non-commutative resolutions of quotient singularities for reductive groups}.
\newblock{\em Invent. Math.} 210 (2017), no. 1, 3--67.


\bibitem{SvdB}
\v{S}.~\v{S}penko, M.~Van den Bergh.
\newblock{Semiorthogonal decompositions for GIT quotients}.
\newblock  {\em Selecta Math.} (N.S.) 27 (2021), no. 2, Paper No. 16, 43 pp.

\bibitem{SvdB2}
\v{S}.~\v{S}penko, M.~Van den Bergh.
\newblock{Non-commutative crepant resolutions for some toric singularities. II}.
\newblock{\em J. Noncommut. Geom.} 14 (2020), no. 1, 73--103.


\bibitem{SvdB3}
\v{S}.~\v{S}penko, M.~Van den Bergh.
\newblock{Comparing the Kirwan and noncommutative resolutions of quotient singularities}.
\newblock \url{https://arxiv.org/pdf/1912.01689.pdf}, 2019.




\bibitem{T}
G.~Tabuada.
\newblock{Noncommutative motives}.
\newblock{University Lecture Series, 63. American Mathematical Society, Providence, RI, 2015. x+114 pp.}



\bibitem{Th}
R.W.~Thomason.
\newblock{Algebraic K-theory and étale cohomology}.
\newblock{\em Ann. Sci. École Norm. Sup.} (4) 18 (1985), no. 3, 437--552.


\bibitem{YZ}
Y.~Yang, G.~Zhao.
\newblock{The cohomological Hall algebra of a preprojective algebra}.
\newblock{\em Proc. Lond. Math. Soc.} (3) 116 (2018), no. 5, 1029--1074.



\end{thebibliography}
\end{document}